\documentclass[a4paper,11pt,leqno,english]{smfart}
\usepackage{aeguill}
\usepackage{enumerate}
\usepackage{amssymb,amsmath,latexsym,amsthm}
\usepackage[T1]{fontenc}
\usepackage{smfthm}      
\usepackage{geometry}   
\usepackage{url}      
\usepackage[frenchb, english]{babel}
\usepackage[utf8]{inputenc}   
\usepackage{mathrsfs}
\usepackage{relsize}
\usepackage{xcolor}
\usepackage{comment}
\definecolor{violet}{rgb}{0.0,0.2,0.7}
\definecolor{rouge2}{rgb}{0.8,0.0,0.2}
\usepackage{hyperref}
\usepackage{mathrsfs}
\hypersetup{
    bookmarks=true,         % show bookmarks bar?
    unicode=false,          % non-Latin characters in Acrobat’s bookmarks
    pdftoolbar=true,        % show Acrobat’s toolbar?
    pdfmenubar=true,        % show Acrobat’s menu?
    pdffitwindow=false,     % window fit to page when opened
    pdfstartview={FitH},    % fits the width of the page to the window
    pdftitle={},    % title
    pdfauthor={},     % author
    colorlinks=true,       % false: boxed links; true: colored links
   linkcolor=violet,          % color of internal links
    citecolor=violet,        % color of links to bibliography
    filecolor=black,      % color of file links
    urlcolor=cyan}           % color of external links
\newcommand{\R}{\mathbb{R}}
\newcommand{\CC}{\mathbb{C}}
\newcommand{\Q}{\mathbb{Q}}

\newcommand{\N}{\mathbb{N}}

\newcommand{\D}{\mathbb D}

\renewcommand{\b}{\bar}
\renewcommand{\d}{\partial}

\newcommand{\vp}{\varphi}

\renewcommand{\O}{\mathcal{O}}

\newcommand{\ep}{\varepsilon}

\newcommand{\la}{\langle}

\newcommand{\ra}{\rangle}

\renewcommand{\ge}{\geqslant}
\renewcommand{\le}{\leqslant}
\renewcommand{\geq}{\geqslant}
\renewcommand{\leq}{\leqslant}
\newcommand{\Ric}{\mathrm{Ric} \,}
\newcommand{\om}{\omega}
\newcommand{\omvp}{\omega_{\varphi}}
\newcommand{\omvpe}{\omega_{\vpe}}

\newcommand{\vpe}{\varphi_{\varepsilon}}
\newcommand{\pse}{\psi_{\varepsilon}}
\newcommand{\Pse}{\Psi_{\varepsilon}}
\newcommand{\phe}{\phi_{\varepsilon}}
\newcommand{\ue}{u_{\varepsilon}}

\renewcommand{\D}{\mathbb D}
\newcommand{\Dlc}{D_{\rm lc}}
\newcommand{\Dklt}{D_{\rm klt}}

\newcommand{\ome}{\om_{\varepsilon}}

\newcommand{\omc}{\omega_{\beta}}
\newcommand{\omb}{\widetilde{\omega}_{\beta}}

\newcommand{\Supp}{\mathrm {Supp}}

\newcommand{\DD}{\mathbb{D}^n\setminus \Delta}
\newcommand{\chid}{\chi_{\delta}}
\newcommand{\uep}{u_{\ep}}
\newcommand{\db}{\partial^{\beta}}
\newcommand{\dbb}{\bar \partial^{\beta}}
\newcommand{\dbj}{\partial_k^{\beta}}
\newcommand{\dbkb}{\partial_{\bar k}^{\beta}}

\newcommand{\cab}{\mathscr{C}^{\alpha, \beta}}
\newcommand{\cdab}{\mathscr{C}^{2,\alpha, \beta}}
\newcommand{\ib}{\int_{B(r)}}

\newcommand{\tr}{\mathrm{tr}}

\newcommand{\pp}{\psi^{+}}
\newcommand{\psm}{\psi^{-}}
\newcommand{\ppm}{\psi^{\pm}}

\newcommand{\iddb}{dd^c}
\newcommand{\ddc}{dd^c}

\newcommand{\fe}{F_{\varepsilon}}

\newcommand{\fepp}{(\varepsilon^2+|z^p|^2)}
\newcommand{\fepq}{(\varepsilon^2+|z^q|^2)}
\newcommand{\fepr}{(\varepsilon^2+|z^r|^2)}
\newcommand{\feps}{(\varepsilon^2+|z^s|^2)}
\newcommand{\fepu}{(\varepsilon^2+|z^u|^2)}
\newcommand{\fept}{(\varepsilon^2+|z^t|^2)}

\newcommand{\dej}{\beta_k}

\renewcommand{\b}{\bar}
\renewcommand{\om}{\omega}

\newcommand\cO{{\mathcal{O}}}

\def\ddbar{\partial\overline\partial}
\let\ol=\overline

\newtheorem*{cor}{Corollary}
\newtheorem*{thma}{Theorem A}
\newtheorem*{thmb}{Theorem B}

\numberwithin{equation}{section}

\begin{document}

\frontmatter 

\title[Conic singularities metrics with prescribed Ricci curvature]{Conic singularities metrics with prescribed Ricci curvature: general cone angles along \\ normal crossing divisors}

\date{\today}
\author{Henri Guenancia}
\address{Institut de Mathématiques de Jussieu \\
Université Pierre et Marie Curie \\
 Paris 
\& Département de Mathématiques et Applications \\
\'Ecole Normale Supérieure \\
Paris}
\email{guenancia@math.jussieu.fr}
\urladdr{www.math.jussieu.fr/~guenancia}

\author{Mihai P\u{a}un}
\address{Korea Institute for Advanced Study\\
School of Mathematics, 85 Hoegiro, Dongdaemun-gu, Seoul 130-722, Korea}
\email{paun@kias.re.kr}

%And this dramatically good paper was written by us, HG and MP.
%It will lay fame and glory upon us till the end of time

\thanks{H.G. is partially supported by the French A.N.R project MACK}

\begin{abstract} 
Let $X$ be a non-singular compact K\"ahler manifold, endowed with an effective 
divisor $D= \sum (1-\beta_k) Y_k$ having simple normal crossing support, and satisfying $\beta_k \in (0,1)$. 
The natural objects one has to consider in order 
to explore the differential-geometric properties of the pair $(X, D)$
are the so-called metrics with conic singularities. In this article, we complete our earlier work \cite{CGP} concerning the Monge-Amp\`ere equations on $(X, D)$ by establishing 
Laplacian and ${\mathscr C}^{2,\alpha, \beta}$ estimates for the solution of this equations regardless to the size of the coefficients $0<\beta_k< 1$. 
In particular, we obtain a general theorem concerning the existence and regularity of Kähler-Einstein metrics with conic singularities along a normal crossing divisor.
\end{abstract}

\maketitle

\tableofcontents

\mainmatter

\section{Introduction}

Let $(X,D)$ be a \textit{log smooth klt pair}, i.e. $X$ is a compact Kähler manifold, and $D= \sum (1-\beta_k)Y_k$ is a $\R$-divisor with simple normal crossing support such that $\beta_k\in (0,1)$ for all $k$.

Given this geometric data, the notion corresponding to a K\"ahler metric in the case $D=0$ 
is Kähler metric with \textit{conic singularities}. For our purposes in this paper, such 
an object is a Kähler metric $\om$ on $X\setminus(\cup Y_k)$ which is quasi-isometric to the model metric with conic singularities: more precisely, near each point $p\in \Supp(D)$ where $\Supp(D)$ is defined by the equation $(z_1 \cdots z_d=0)$ for some holomorphic system of coordinates $(z_i)$, we want $\om$ to satisfy
\[C^{-1} \om_{\rm cone} \le \om \le C \om_{\rm cone}\]
for some constant $C>0$, and where 
\[\om_{\rm cone}:=\sum_{k=1}^d \frac{1}{|z_k|^{2(1-\beta_k)}}\sqrt{-1}dz_k\wedge d\bar z_k +\sum_{k=d+1}^n \sqrt{-1}dz_k\wedge d\bar z_k \] 
is the model cone metric with cone angles $2\pi \beta_k$ along $(z_k=0)$.

\noindent
Given a log smooth klt pair $(X,D)$, a natural question to ask is whether one can find a \textit{Kähler-Einstein} metric $\om$ on $X\setminus \Supp(D)$ (i.e. satisfying on this open subset $\Ric \om = \lambda \om$ for some $\lambda \in \R$) having conic singularities along $D$. Such a metric will be refered to as a \textit{conic Kähler-Einstein metric}.\\

In this paper, we provide a complete (positive) answer to this question (Theorem A) and we also derive finer regularity estimates for these metrics by proving a conic analogue of Evans-Krylov theorem for complex Monge-Ampère equations (Theorem B).
%In the case where the divisor $D$ is smooth (i.e. is the union of \textit{disjoint} hypersurfaces), and using the "openness" part proved by Donaldson \cite{Don}, the existence of Kähler-Einstein metrics was obtained by Brendle \cite{Brendle} when $\beta \in (0, 1/2]$; it was also independently proved by 
%Jeffres-Mazzeo-Rubinstein \cite{JMR} without conditions on the cone angle. Also at the same time, we dealt in \cite{CGP} with the more general case where $D= \sum (1-\beta_j) Y_j$ has simple normal crossing support, but still assuming $\beta_j \in (0, 1/2]$ for all $j$ in order to bound from below the curvature of our model.
%Here, we intend to explain how to bypass the difficulty of working with unbounded curvature so as to get rid of the assumption on the angles. We explain the new technical input of this article in the last paragraph of the section concerning the strategy of proof.\\

%Let us now give some more details so as to state our results properly. 
Actually, the results we obtain are very similar to the classical case $D=0$.
We remark that in order to expect the pair $(X, D)$ to admit a Kähler-Einstein conic metric 
the necessary condition is  
%there are cohomological obstructions to the existence of such a , involving 
an appropriate positivity property for the adjoint $\R$-line bundle $K_X+D$. When this
requirement is fulfilled, one can show that any such metric is necessarily the solution of a \textit{global} Monge-Ampère equation of the form
\[(\om+\ddc \vp)^n = \frac{e^{\mu \vp}dV}{\prod |s_k|^{2(1-\beta_k)}} \leqno{\rm (MA)}\] 
where $\om$ is a background Kähler metric on $X$, $\mu \in \R$ is a parameter which could be related to the sign of the curvature, $dV$ is some suitable smooth volume form on $X$, and $s_k$ are sections of $\cO(Y_k)$ defining the hypersurface $Y_k$; finally, $\vp$ is a \textit{bounded} $\om$-psh function.

\noindent
If $dV$ is chosen according to the cohomological positivity properties of $K_X+ D$, then a solution $\omvp:=\om+\ddc \vp$ of (MA) satisfies \[\Ric \omvp = -\mu \omvp + [D] \leqno {\mathrm{(KE)}}\] where $\Ric \omvp := -\ddc \log \omvp^n$ (it is automatically well-defined as a current). Note that such equations with meromorphic right hand side were first considered and solved (under some assumptions) by Yau, cf \cite[\S 8]{Yau78}.\\

\noindent
Hence in order to construct Kähler-Einstein conic metrics a first step would be to solve the equation (MA). We remark that it is a priori not clear that a solution of (MA) will have conic singularities along 
$D$ \--- even if by the general theory the function $\varphi$ is smooth outside of the support of the divisor \---. Indeed, the equations (KE) or (MA) only impose the behavior of the \emph{determinant} of the metric $\omvp$ whereas having ``conic singularities" is a much more precise information about the metric itself. Nevertheless, we have the following statement. 

\begin{thma}
Let $(X,D)$ be a log smooth klt pair with $D= \sum (1-\beta_k)[s_k=0]$. Let $\om$ be a Kähler metric on $X$, $dV$ a smooth volume form, and let $\mu \in \R$. Then any weak solution $\omvp=\om + \ddc \vp$ with $\vp \in L^{\infty}(X)$ of 
\[(\om+\ddc \vp)^n = \frac{e^{\mu \vp}dV}{\prod |s_k|^{2(1-\beta_k)}} \] 
has conic singularities along $D$. \\
\end{thma}
\noindent
This result indicates that the restriction of the solution $\om + \ddc \vp$ to any coordinate set
has the same singularities as the local
model metric $\om_{\rm cone}$. 

As a conclusion, in order to construct Kähler-Einstein conic metrics, it is enough to produce weak solutions of (MA). In the case of non-positive curvature, this is essentially a consequence of S.~Kolodziej's theorem \cite{Kolo}. In the positively curved case however, such metrics do not always exist, but there is a criterion involving the properness of the log-Mabuchi functional guaranteeing its existence (see \cite{rber} or \cite{BBEGZ} for a generalization to the general setting of (singular) log Fano varieties); cf also \cite{JMR} for the existence of positively curved conic KE metric under that properness assumption, $D$ being smooth. 

\begin{cor}
Let $(X, D)$ be a log smooth klt pair. 
\begin{enumerate}
\item[$(i)$] If $K_X+D$ is ample, then there exists a unique conic Kähler-Einstein metric with negative curvature.
\item[$(ii)$] If $K_X+D$ is numerically trivial, then there exists in each Kähler class a unique conic Ricci-flat metric.
\item[$(iii)$] If $-(K_X+D)$ is ample and the log-Mabuchi functional is proper, then there exists a unique conic Kähler-Einstein metric with positive curvature.
\end{enumerate}
\end{cor}

The previous result was obtained a few years ago by \cite{Brendle, CGP, JMR} independently under some various additional assumptions, and led to several further works \cite{LS, SW}. 
S.~Brendle assumed that the support of $D$ is smooth (i.e. is the union of \textit{disjoint} hypersurfaces), and it satisfies $\beta \le 1/2$; in \cite{JMR} the smoothness assumption on $D$ was present too, but they had no restriction concerning the coefficient $\beta$. And in our previous work \cite{CGP}, the above result was established under the assumption that $\beta_k \le 1/2$ for all $k$. 
We note that the condition above is automatically satisfied in the orbifold case, and that it needed in a crucial way so as to bound the holomorphic bisectional curvature of the cone metric outside the aforesaid hypersurface. However, as the spectacular results in \cite{CDS1, CDS2, CDS3}  and \cite{T} show, it is important to dispose of this kind of results in full generality i.e. without any restriction on the size of the coefficients.
Finally, let us mention two papers that appeared after the first version of this article was released: Yao \cite{Yao} gave a new approach to the Laplacian estimate (when $D$ is smooth) by localizing the problem and cleverly running a Moser iteration scheme; and around the same time, Datar-Song \cite{DSong} showed how to deduce the Laplacian estimate in the normal crossing case from the smooth case using a regularization argument.

%Nouveau paragraphe, car on ne m�lange pas les serviettes et les torchons...
In \cite{MR}, Mazzeo-Rubinstein announced the general case of a log smooth divisor; as far as we understand, the method they seem to use is very different from the one which will be presented here.\\

As a consequence of the Corollary above we establish the vanishing/parallelism of orbifold holomorphic tensors in the sense of Campana (cf. Theorem \ref{thm:van}). Actually, for this geometric 
application the quasi-isometry properties of conic singularities metrics established in Theorem A are sufficient, i.e. higher regularity of the metric (as in Theorem B) is not required. We remark that this is equally the case in many other contexts involving metrics with conic singularities
(e.g. the generic semi-positivity of the log cotangent bundle of the pairs $(X, D)$ with pseudo-effective canonical class, stability of the tangent bundle of singular varieties whose canonical bundle is ample...).

Next, we show that Theorem A can be extended to general klt pairs (Theorem \ref{thm:klt}) and to log smooth log canonical pairs (Theorem \ref{thm:lc}), in the spirit of \cite{G2, G12}.
More precisely, we prove that any Kähler-Einstein metric corresponding to a klt pair $(X,D)$ has conic singularities along $D$ on the so-called log smooth locus of the pair $(X,D)$, which is the Zariski open subset of $X$ consisting of points around which the pair is log smooth (i.e. $X$ is smooth and $D$ has simple normal crossing support). 

We also prove that any Kähler-Einstein metric on a log smooth log canonical pair $(X,D)$ (i.e. the same setting as Theorem A, but allowing some $\beta_j$'s to be zero) has mixed cone and cusp singularities along $D$.\\

Finally, we investigate in the last part of our paper the question of higher regularity for the potential $\vp$ solution of equation (MA). In \cite{Don}, Donaldson introduced Hölder spaces adapted to the conic setting (we refer to section \ref{sec:func} for the definitions of these spaces). 
From our point of view, it is more natural to work with functions whose higher regularity properties are 
modeled after the notions appearing in the theory of orbifolds. For example, one can define a ``co-tangent space"
$T^\star_X\langle B\rangle$ associated to the klt pair $(X, B)$, see e.g. \cite{CP}. In order to go on and construct the
\emph{exterior differential operator}, simple examples show that one has to allow forms whose coefficients can be expanded in Puiseux series (rather than Taylor) near the support of $B$. The regularity notions for these objects are defined in terms of local ramified coverings, and we introduce the spaces $\cab$ and $\cdab$ in a similar way, so as to have a certain
coherence between algebraic and differential geometry of $(X, B)$. 
\medskip

\noindent The following result, relying on Theorem A, should be viewed as a conic version of Evans-Krylov theorem for complex Monge-Ampère equations: 
\begin{thmb}
Let $(X,D)$ as in Theorem A, and let $\vp\in L^{\infty}(X)$ be any solution of 
\[(\om+\ddc \vp)^n = \frac{e^{\mu \vp}dV}{\prod |s_k|^{2(1-\beta_k)}} \] 
Then $\vp$ belongs to the class $\cdab$. 
\end{thmb}

This kind of result has been already studied before in the particular case where $D$ is smooth. 
More precisely, Brendle proved it whenever $\beta \le 1/2$ by adapting the original $\mathscr C^3$ estimate of Aubin-Yau to the conic case; however this approach needs the curvature of the model metric to be bounded. In \cite{JMR}, the author's approach is using edge calculus, while in \cite{CDS2}  the argument is based on Schauder estimates, well adapted to their precise context.
\medskip
 
We propose here a new approach in the normal crossing setting, based on branched covers as in \cite{CP} so as to mimic Evans-Krylov theory in the non-degenerate case. However, several serious issues have to be addressed as we will briefly explain at the end of the following paragraph. 

\vspace{3mm}
\bigskip

\noindent
\textbf{Overview of the arguments.}

\noindent
We now discuss briefly the ideas in our proof of Theorem A. We will proceed as in \cite{CGP}: we regularize the equation $(\mathrm{MA})$ by introducing the following family of non-degenerate Monge-Amp\`ere equations:
\[\om_{\vpe}^n=\frac{e^{\mu\vpe}dV}{\prod_{j=1}^d(\ep^2+|s_j|^2)^{1-\beta_j}}\leqno \mathrm{(MA}_{\ep})\]
which has to be suitably normalized if $\mu=0$, and a bit modified if $\mu<0$ (cf \S \ref{sec:reg}).
Using a stability argument, we observe that the solution $\omvpe$ of this equation will converge to the initial solution $\omvp$ of (MA). 
Therefore, in order to achieve our goal, it would be enough to obtain uniform estimates for $\omvpe$ with respect to some approximation $\ome$ of a reference conic metric. It is important in the process that $\omega_{\varepsilon}:= \om+ \iddb \pse$ belongs to the fixed cohomology class $[\om]$; the explicit expression of $\psi_\ep$ is given in \S \ref{sec:reg2}.\\

The first step is to use the results of \cite{Kolo} to derive $\mathscr C^0$ estimates; this combined with standard results in the theory of Monge-Amp\`ere equations gives us \textit{interior} $\mathscr C^{2, \alpha}$ estimates \emph{provided that} global  
laplacian estimates have been already established. If we fulfill this program, then we can extract from $(\om_{\vpe})_{\ep>0}$ a subsequence converging to the desired solution $\omvp$, which will henceforth be smooth outside the support of $D$. \\

Next, we aim to compare $\omvp$ and $\om_{\rm cone}$; to this end we will show that 
$\displaystyle \omega_{\varphi_\ep}$ and $\ome$ are uniformly 
quasi-isometric (with respect to $\varepsilon$). It is important to realize that in our general situation and unlike in \cite{CGP}, the uniform lower bound on the holomorphic curvature of $\om_{\ep}$ \emph{does not holds in general}.
This quantity is usually needed in order to get the estimates. The new idea 
in this article is that by introducing a bounded function of type
$$C\sum_k| s_k|^{2\rho}$$
under the Laplacian $\displaystyle \Delta_{\omega_{\varphi_\varepsilon}}$ appearing in Siu-Yau's inequality, we are able to compensate the singularity arising from the curvature tensor, and proceed as in the classical case (see e.g. \cite[Proposition 2.1]{CGP}).
Here $C> 0$ will be a (large) positive constant, and $0<\rho< 1$ a (small) parameter to be chosen. 
We may add that a similar trick appears in \cite{Brendle} to deal with order three estimates, although in his case the curvature is
bounded and the situation is far less delicate.\\ Actually we will formulate here a general and intrinsic Laplacian estimate by replacing the usual lower bound hypothesis for the curvature tensor
with the condition that the said tensor is bounded from below by the $\ddc$ of a \textit{bounded} function, cf. Proposition \ref{prop:key}. Of course, the Hessian of the function we consider must be compatible with the
rest of the geometric data involved in the equation, but we will see that this can be achieved in the 
context of Theorem A.\\

As for Theorem B, our proof relies on an adaptation of Evans-Krylov theory to the conic setting.
Roughly speaking, we first treat the case of rational coefficients and then we obtain the general case by a limit process, using the uniformity of the estimates we establish in the rational setting. 

In order to overcome the difficulty
induced by the non-ellipticity of the conic laplacian we will consider branched covers of the coordinate open sets of $X$. The motivation for introducing such covering maps is as follows.
The property we have to establish involves differentiation with respect to multi-valued vector fields of type $\displaystyle z^{1-\beta}\frac{\partial}{\partial z}$. The observation is that this type of vector fields become single-valued (but meromorphic, in general) on a branched cover, provided that  the ramification is chosen according to the denominators of $\beta_j$.

Next, we recall that at the heart of Evans-Krylov's argument lies the weak Harnack inequality; we establish here a very precise  version of this result for the pull-back of the cone metric by the branched cover.
This is the most delicate part of our proof, in particular because the specificity of Evans-Krylov's method compels us to work not only on geodesic balls in $\CC^n \setminus \Delta$ but also on balls centered at a point of $\Delta$. Concerning the technical tools we establish in this part of our paper we mention the Sobolev inequality, and the integration by parts formula ``in conic setting".
This later technique is a bit non-standard, as it involves functions which are subharmonic with respect to a metric with conic singularities, rather than plurisubharmonic functions.\\

%The main difficulty is that Siu-Yau's inequality is essentially expressed in terms of geodesic coordinates for $\ome$ while we only know a precise expression of the curvature tensor of $\ome$ in a coordinate system involving the geometry of $(X,D)$. So we will have to juggle with those two different coordinate systems, which involves relatively heavy computations in particular because of the various components of the divisor. 

\noindent
\textbf{Organization of the paper.}
\begin{itemize}
\item[$\bullet$] \S 2: We prove here a general Laplacian estimate in a framework including some geometries with unbounded curvature, like typically the conical one.
\item[$\bullet$] \S 3: We collect some facts from \cite{CGP}: the construction of the regularized conic metric $\ome$, and the expression of its curvature tensor in some coordinate system adapted to the geometry of the pair $(X,D)$. We will observe that the curvature of $\ome$ cannot be uniformly (in $\ep$) bounded below, which is the main source of issues. 
\item[$\bullet$] \S 4: We introduce a particular type of uniformly bounded smooth functions, denoted by $\Psi_{\ep}$, whose $\ddc$ compensates the singularities of the curvature tensor of $\ome$, so that it can be used as an auxiliary function in the general estimate established in \S 2.
\item[$\bullet$] \S 5: We establish various estimates related to our Monge-Ampère equation in order to be able to apply the Laplacian estimate of \S 2, and conclude the proof of Theorem A.
\item[$\bullet$] \S 6: As an application of Theorem A, we get a vanishing theorem for orbifold holomorphic tensors; we also generalize Theorem A to general klt pairs and log smooth log canonical pairs.  
\item[$\bullet$] \S 7: We prove Theorem B, namely the Hölder estimates for the second derivatives of the potential of the conic Kähler-Einstein metric.\\
\end{itemize}

\noindent
\textbf{Acknowledgements.}
We are grateful to Sébastien Boucksom for his insightful suggestions which helped a lot improve the exposition of the present article. We also thank Tien-Cuong Dinh for sharing his valuable ideas regarding section \ref{sec:ipp}, as well as Jianchun Chu for pointing out a small inaccuracy in the previous version of this work.

\noindent
Part of this work was completed during the first author's visit to the Korea Institute for Advanced Study, and the
second author's visit to the Hong-Kong University, the National Taiwan University and
the Institute for Mathematical Sciences (Singapore), respectively. M.P. is grateful to
Ngaiming Mok, Jungkai Alfred Chen and Wing Keung To for the invitation, and for the excellent working conditions
provided by these institutes.

\section{Estimates for the Monge-Amp\`ere operator}
\label{sec:estimates}
\medskip
Let $(X, \omega)$ be a $n$-dimensional compact 
complex manifold, endowed with  
K\"ahler metric. The following Laplacian estimate is a generalization of the usual estimate due to Yau \cite{Yau78} (see also \cite{Siu, Paun, CGP, BBEGZ}) involving a lower bound of the holomorphic bisectional curvature of $\om$. Here we allow (negative) degeneracy of the curvature as long as it is controlled by the $\ddc$ of a bounded function. More precisely, we have

\begin{prop}
%\phantomsection
\label{prop:key}
Let $\omvp:=\om+\ddc \vp$ be a Kähler metric satisfying $$\omvp^n= e^{\pp-\psm} \om^n$$ for some smooth functions $\ppm$. We assume that there exists $C>0$ and a smooth function $\Psi$ such that:
\begin{enumerate}
\item[$(i)$] $\sup_X |\vp| \le C$
\item[$(ii)$] $\ddc \Psi \ge -C \om$ and $\, \sup_X |\Psi|\le C$
 \item[$(iii)$] $\ddc \ppm \ge -(C\om + \ddc \Psi)$ and $\, \sup_X |\ppm| \le C$
\item[$(iv)$] $i\Theta_{\om}(T_X) \ge -(C\om + \ddc \Psi) \otimes \mathrm{Id}$
\end{enumerate}
Then there exists a constant $A>0$ depending only on $C$ such that 
\[A^{-1} \om \le \omvp \le A \om\]
\end{prop}
\bigskip

Here $\Theta_{\om}(T_X)$ denotes the Chern curvature tensor of $(T_X, \om)$, and the inequality in $(iv)$ is to be taken in the sense of Griffiths positivity. 

The new feature in this statement lies in the introduction of the function $\Psi$ which is only supposed to be uniformly bounded and uniformly quasi-psh (cf $(ii)$). For example, the case $\Psi=0$ in condition $(iv)$ would just mean that the holomorphic bisectional curvature of $\om$ is bounded from below by $C$. So this more general framework enables more flexibility compared with the usual case $\Psi=0$ (cf $(iii)$ and $(iv)$), and this will be crucial in our matter.

\begin{proof}
We divide the proof into three steps. The first one consists in recalling the usual Siu-Yau laplacian inequality; in the second, we will deal with the singularities coming the curvature tensor, and in the last one, we will take care of the singular terms involving laplacians of $\ppm$. \\

\noindent
\textbf{Notations.}

\noindent
For now, we do not need to know that the rhs $e^{\pp-\psm} \om^n$ has a special form, so we will set $f:= \pp-\psm$. We denote by $g$ (resp. $g_{\vp}$) the hermitian metric on $T_X$ induced by $\om$ (resp. $\omvp$). If $\Theta_{\om}(T_X)$ the Chern curvature of $(T_X, \om)$, then $i\Theta_{\om}(T_X)$ is a real $(1,1)$-form with values in the bundle of hermitian endomorphisms of $T_X$. Its contraction with $\om$, that we will write  $i\widetilde \Theta_{\om}$, is thus naturally a $(1,1)$-form with values in the bundle of hermitian endomorphisms of $T_X^*$. We will denote by $g_{\vp}^{-1}$ the hermitian metric induced by $g_{\vp}$ on $T_X^\star$.\\

\begin{comment}
In local coordinates, we have
\[\omega= \sqrt{-1}\sum_{i, j}g_{i\ol j} \, dz^i \wedge dz^{\ol j}\]
as well as
\[\om_\vp= \sqrt{-1}\sum_{i, j}g^\prime_{ i\ol j} \, dz^i \wedge dz^{\ol j}\]
The components of the curvature tensor of $\om$ are given by
\[R_{ i\ol j k \ol l }:= - {\partial ^2{g_{i\ol j}}\over \partial z^k\partial z^{\ol l}}+
\sum_{s, r}g^{s\ol r}{\partial g_{ s\ol j}\over \partial z^{\ol l}}
{\partial g_{ i\ol r}\over \partial z^{k}}.\]
The components of the tensor $\big(R^{i\ol j}_{ k\ol l}\big)$ are obtained from 
$\big(R_{ i\ol j k \ol l}\big)$ by contraction with the metric $\omega$, and 
\[R_{ i\ol j}:= \sum_{p, q}g^{p\ol q} R_{ i\ol j p\ol q}\]
are the coefficients of the Ricci curvature of $\om$ in the $(z)$--coordinates.\\
\end{comment}

\noindent
\textbf{Step 1: The Laplacian inequality.}

\noindent
We recall the following result, extracted from \cite[(3.2) p. 99]{Siu}.

\begin{prop} 
\label{ineq}
We have the following inequality
\[\Delta_{\omega_\vp} (\log\tr_{\om} \om_\vp ) \ge \frac{1}{\tr_{\om}\om_\vp} \left[ -\tr_{\om} \Ric(\omvp)+ \tr_{\omvp}\left(\tr_{g_{\vp}^{-1}}(i\widetilde \Theta_{\om})\right)  \right]\]
\end{prop}

\noindent We remark that in the last term of the relation above we take the trace $\tr_{\omvp}$ of the contravariant part of the curvature tensor, and then take the trace $\tr_{g_{\vp}^{-1}}$ of the covariant part.

Let $p\in X$ be an arbitrary point; we consider a coordinate
system 
$w= (w^1,\ldots, w^n)$ on a small open set containing $p$, such that 
$\omega$ is orthonormal
and such that $\om_\vp$ is diagonal at $p$ when expressed in
the $w$-coordinates, i.e.
\begin{equation*}
\om_\vp= \sqrt{-1}\sum \lambda_j dw^j\wedge dw^{\ol j}
\end{equation*} 

\noindent
If we denote by $(R_{ i\ol j k \ol l})$ the components of $i\Theta_{\om}(T_X)$ with respect to the $w$-coordinates, we have:
\begin{equation*}
-\tr_{\om} \Ric(\omvp) = \Delta_{\om}f-\sum_{i,k}R_{i\bar i k \bar k}
\end{equation*}
as well as:
\begin{equation*}
\tr_{\omvp}\left(\tr_{g_{\vp}^{-1}}(i\widetilde \Theta_{\om})\right) = \sum_{i,k} \frac{\lambda_i}{\lambda_k}R^{i\bar i }_{k \bar k}= \sum_{i,k} \frac{\lambda_i}{\lambda_k}R_{i\bar i k \bar k}
\end{equation*}
Therefore, combining these two equalities with Proposition \ref{ineq}, we get:
\begin{equation}
\label{eq:ob4}
\Delta_{\omega_\vp} (\log\tr_{\om} \om_\vp ) \ge
\frac{1}{\sum_p \lambda_p} \left(\sum_{i\le k} \left({\lambda_i\over 
\lambda_k}+ {\lambda_k\over \lambda_i}- 2\right) R_{ i\ol i k \ol k}(w)+ \Delta_\om f \right)\\
\end{equation}
\bigskip

\noindent
\textbf{Step 2: Dealing with the curvature.}

\noindent
We are now going to exploit assumption $(iv)$. Recall that a form $\alpha \in \Omega^{1,1}_X(\mathrm{End}(T_X))$ is said to be Griffiths semipositive, what we write $\alpha \ge 0$, if for any vector fields $u,v$, we have $\la\alpha(u,u)v,v\ra_{\om} \ge 0$. 
So in our case, we can rewrite condition $(iv)$ in the previously chosen geodesic coordinates as

\[R_{i\bar jk \bar l}\, u_i\bar u_j v_k \bar v_l \ge -(C\delta_{i\bar j}+\Psi_{i\bar j})u_i\bar u_j |v|^2_{\om} \]
where $\Psi_{i \bar j} = \frac{\d^2 \Psi}{\d w_i \d \bar w_j}$.

\noindent
Applying this inequality with $u,v$ vectors of the orthornormal basis, we obtain for all $i,k$:
 \[R_{i \bar i k \bar k} \ge -(C+ \Psi_{i\bar i})\] 
Using the symmetries of the curvature tensor, we also get $R_{i \bar i k \bar k} \ge -(C+ \Psi_{k\bar k})$.\\

\noindent
We claim that \[\Delta_{\omvp} \Psi \ge \frac{-1}{\sum_p \lambda_p} \sum_{i< k}\left(\frac{\lambda_i}{\lambda_k}+\frac{\lambda_k}{\lambda_i}-2\right) R_{i \bar i k \bar k} -C \tr_{\omvp}\om \leqno{(*)}\] 
for some $C>0$.

\noindent
To show $(*)$, we use the previous inequalities which yield:
\[ \frac{1}{\sum_p \lambda_p}\left(\frac{\lambda_i}{\lambda_k}+\frac{\lambda_k}{\lambda_i}-2\right) R_{i \bar i k \bar k} \ge - \frac{1}{\sum \lambda_p}\left[\frac{\lambda_k}{\lambda_i}(C+\Psi_{i\bar i})+\frac{\lambda_i}{\lambda_k}(C+ \Psi_{k\bar k})\right] \]
As $C+\Psi_{i\bar i} \ge 0$ for all $i$ by assumption $(ii)$, we have: 
\begin{eqnarray*}
\Delta_{\omvp}\Psi &=& \sum_i \frac{1}{\lambda_i} (C+\Psi_{i\bar i})-C \tr_{\omvp}\om \\
&\ge&  \frac{1}{\sum \lambda_p}\sum_{i< k }\left[\frac{\lambda_k}{\lambda_i}(C+\Psi_{i\bar i})+\frac{\lambda_i}{\lambda_k}(C+ \Psi_{k\bar k})\right] -C \tr_{\omvp}\om  
\end{eqnarray*}
which shows $(*)$. Combining \eqref{eq:ob4} with $(*)$, we finally obtain
\begin{equation}
\label{eq:ob6}
\Delta_{\omega_\vp} (\log\tr_{\om} \om_\vp+\Psi ) \ge \frac{\Delta_\om f}{\tr_{\om} \omvp} - C \tr_{\omvp}\om 
\end{equation}
\medskip

\noindent
\textbf{Step 3: End of the proof.}

\noindent
The last term to deal with is $\Delta f$. Recall that $f = \pp- \psm$. 
By assumption $(iii)$, \[\Delta \pp \ge -nC - \Delta_{\om} {\Psi}\] and as 
\[\Delta_{\omvp} \Psi = \sum_i \frac{1}{\lambda_i} \Psi_{i \bar i } \ge \frac{\Delta \Psi}{\sum_p \lambda_p}- C \tr_{\omvp}\om\] we get 
\begin{equation}
\label{eq:ob7}
\Delta_{\omvp} \Psi  \ge -\frac{\Delta \pp}{\sum_p \lambda_p}-nC \tr_{\omvp}\om.
\end{equation}
Let us now treat the term $-\Delta \psm$. By assumption $(iii)$, we have  \[C\om+\ddc (\Psi+\psm)  \le \tr_{\omvp}(C\om+\ddc( \Psi+\psm)) \,  \omvp\] and by taking the trace with respect to $\om$, we get
\begin{equation}
\label{eq:ob8}
\Delta_{\omvp}  (\Psi+\psm) \ge -C \tr_{\omvp}\om + \frac{\Delta(\Psi+ \psm)}{\tr_{\om}\omvp}
\end{equation}
Plugging $(ii)$, \eqref{eq:ob7} and \eqref{eq:ob8} into \eqref{eq:ob6}, we obtain:
\begin{equation}
\label{eq:ob9}
\Delta_{\omega_\vp} (\log\tr_{\om} \om_\vp+3\Psi +\psm) \ge - C \tr_{\omvp}\om 
\end{equation}
for some bigger constant $C$.\\

\noindent
Now, as $\Delta_{\omvp} \vp = n - \tr_{\omvp} \om$, we have:
\[\Delta_{\omega_\vp} (\log\tr_{\om} \om_\vp+3\Psi +\psm-(C+1) \vp) \ge \tr_{\omvp}\om-n(C+1) \]
and we can apply the maximum principle as usual. As we have a priori bounds on $\Psi, \psm$ and $\vp$ by assumption, we obtain the desired result.
\end{proof}

\section{Metrics and Curvature Tensors}

In this section we will collect a few facts from \cite{CGP} concerning the
construction of metrics adapted to the pair $(X, D)$, their approximations and their corresponding curvature tensor.

\subsection{The regularized metric}
\label{sec:reg2}

Recall that $D=\sum_{k=1}^d (1-\beta_k)Y_k$ is a divisor with simple normal crossing support. For each $k=1, \ldots, d$, we can choose a section $s_k$ of $\cO(Y_k)$ cutting out the hypersurface $Y_k$; we also fix a smooth hermitian metric $h_k$ on this line bundle.

 In order to construct a sequence of regularized cone metrics (with respect to $D$), we introduce for any $\ep \ge 0$ the functions
 $\chi_k= \chi_{k, \varepsilon}: [\varepsilon^2, \infty[\to \R$ defined as follows:  
\begin{equation}
\label{3}
\chi_k(\varepsilon^2+ t)= {1\over \beta_k}\int_0^t{(\varepsilon^2+ r)^{\beta_k}- \varepsilon^{2\beta_k}\over r}dr
\end{equation}
for any $t\ge 0$. There exists a constant $C> 0$ independent of $k, \varepsilon$ such that
$0\le \chi_k(t) \le C$
provided that $t$ belongs to a bounded interval. Also, for each $\varepsilon> 0$
the function defined in \eqref{3} is smooth.\\

The choice of the function $\chi_k$ above is motivated by the following equality:
\begin{equation}
\label{5}
i\ddbar\chi_k\big(\varepsilon^2+ |s_k|^2\big)=
\sqrt{-1}{\langle D^\prime s_k, D^\prime s_k\rangle\over 
(\varepsilon^2+ |s_k|^2)^{1- \beta_k}}
- \frac{1}{\beta_k} \big((\varepsilon^2+ |s_k|^2)^{\beta_k}-\varepsilon^{2\beta_k}\big)\Theta_k
\end{equation}
where $D^\prime$ the $(1, 0)$ part of the Chern connection associated to $(\cO(Y_k), h_k)$, and by $\Theta_k$ the curvature form of $(\cO(Y_k), h_k)$.\\

Let $\omega$ be any K\"ahler metric on $X$; we consider the $(1,1)$-form 
\[\omega_{ \varepsilon}:= \omega+ {1\over N}\sum_{k=1}^d i\ddbar \chi_k\big(\varepsilon^2+ |s_k|^2\big)\] on $X$. For $N$ big enough (independent of $\ep$), we have $\ome \ge \om/2$, so that $\ome$ is a Kähler form. We fix such a $N$ till the end of this article. We denote by 
\begin{equation}
\label{pse}
 \psi_\varepsilon:= {1\over N}\sum_{k=1}^d \chi_k\big(\varepsilon^2+ |s_k|^2\big)
 \end{equation}
the potential; it satisfies $\sup_X |\pse| \le C$ for some uniform (in $\ep$) constant $C>0$. \\

Let now $p\in X$, and $p\in \cap (s_k=0)$ for $k= 1,\ldots, d$. We define the $(z)$-coordinates by the local expressions of $s_k$,
completed in an arbitrary manner. Then we have
\[\omega_\varepsilon|_\Omega\ge C\sum_k\sqrt{-1}{dz^k \wedge d\bar z^{ k}\over (\varepsilon^2+ |z^k|^2)^{1-\beta_k}}\]
Therefore, if $u= \sum_i u_i \frac{\d}{\d z^i}$ is a vector field of norm $1$, we have for all $k\in \{1, \ldots, d\}$:
\begin{equation}
\label{obs}
|u_k|^2 \le (\ep^2+|z^k|^2)^{1-\beta_k}
\end{equation}

\begin{comment}
hence we have
\[|dz^i|_{\ome}^2\leq C (\varepsilon^2+ |z^i|^2)^{1-\beta_i} ;\]
therefore we have the estimate
\begin{equation}
\label{in7}
\sum_k\Big|{\partial z^i\over \partial w^k}\Big|^2 
\le C(\varepsilon^2+ |z^i|^2)^{1-\beta_i}
\end{equation}
where $(w)$ is some geodesic coordinate system for $\ome$, and the constant $C$ above is independent of $\varepsilon$.\\
\end{comment}

\subsection{The curvature tensor of $\ome$}

Now, we recall the computation of the components of the curvature tensor
in \cite{CGP}. First recall the following elementary (and standard) result (see e.g. \cite[Lemma 4.1]{CGP} for the proof), which will provide us with a coordinate system adapted to the pair $(X, D)$.

\medskip 

\begin{lemm}
\label{lem:coor}
Let $(L_1, h_1),\ldots, (L_d, h_d)$ be a set of hermitian
line bundles, and for each index $k= 1,\ldots,d$, let $s_k$ be a section of $L_k$; we assume that the hypersurfaces 
$Y_k:= (s_k= 0)$
are smooth, and that they have strictly normal intersections. 
Let $\displaystyle p_0\in \cap Y_k$; then there exist a constant $C>0 $ and an open set $\Omega\subset X$ centered at $p_0$, such that for any point $p\in \Omega$ there exists a coordinate system $z= (z^1,\ldots, z^n)$ and a trivialization $\tau_k$ for $L_k$ such that:

\begin{enumerate}
\item [$(i)$] For $k= 1,\ldots, d$, we have $Y_k\cap \Omega= (z^k= 0)$;

\item [$(ii)$] With respect to the trivialization $\tau_k$, the metric
$h_k$ has the weight $\varphi_k$, such that 
\begin{equation}
\label{eq:coor}
\varphi_j(p)= 0, \quad d\varphi_j(p)= 0, \quad 
\Big|{\partial^{|\alpha|+|\beta|}\varphi_j\over \partial z^{\alpha}\partial z^{\bar \beta}}(p)\Big|\le C_{\alpha, \beta}
\end{equation}
for some constants $C_{\alpha, \beta}$ depending only on the multi indexes $\alpha, \beta$.
\end{enumerate}
\end{lemm}

\medskip
In these coordinates, we have the following explicit expression of the coefficients $(g_{p\b q })$ of the
metric $\ome$ constructed in \S \ref{sec:estimates}:
\begin{eqnarray}
g_{p\b q }&= &\, \omega_{p\b q}+ e^{-\varphi_q}{\delta_{pq}+ z^q\alpha_{pq}
\over (\ep^2+ |z^q|^2e^{-\varphi_q})^{1- \beta_q}}+ 
e^{-\varphi_p}{z^{\b p}\overline {\alpha_{qp}}
\over (\ep^2+ |z^p|^2e^{-\varphi_p})^{1- \beta_p}}+ \label{eq:metr} \nonumber \\
&+ & \sum_k{|z^{k}|^2 {\beta_{kpq}}
\over (\ep^2+ |z^k|^2e^{-\varphi_k})^{1- \beta_k}}+
((\ep^2+ |z^k|^2e^{-\varphi_k})^{\beta_k}-\ep^{2\beta_k})\frac{\partial^2 \vp_k}{\partial z^p \partial \bar z^q} \nonumber
\end{eqnarray}
 
\noindent The expressions $\alpha, \beta$ above are functions of the partial derivatives of $\varphi$; in particular, $\alpha$ is vanishing at the given point $p$
at order at least 1, and the order of vanishing of $\beta$ at $p$ 
is at least 2 (this helps a lot in the computations to follow...).\\

\noindent We recall the following useful result; for the proof, we refer to \cite[Lemma 4.2]{CGP}.

\begin{lemm}
\label{lem:inv}
In our setting, and for $\ep^2+|z|^2$ sufficiently small, we have:
\begin{enumerate}
\item[$(i)$] For every $k \in \{1, \ldots, d\}$, \[g^{k\b k}(z)=(\ep^2+|z^k|^2)^{1-\beta_k}(1+\O((\ep^2+|z^k|^2)^{1-\beta_k}));\]
%Moreover, we have the inequality $g^{k\b k}(p)\le (\ep^2+|z^k|^2)^{1-\beta_k}$. 
\item[$(ii)$] For every $k,l \in \{1, \ldots, d\}$ such that $k\neq l$, \[g^{k \b l}(z)=\O((\ep^2+|z^k|^2)^{1-\beta_k}(\ep^2+|z^l|^2)^{1-\beta_l}),\]
\end{enumerate}
the $\O$ being with respect to $\ep^2+|z|^2$ going to zero.
\end{lemm}

We will use the two previous lemmata in order to analyze the singularity of the curvature tensor corresponding to the vector bundle $\displaystyle (T_X, \om_\varepsilon)$ as 
$\varepsilon\to 0$. We will evaluate separately the quantities 
$\displaystyle \frac{\partial^2 g_{p\b q}}{\partial z^r \partial \bar z^s}$ and
$\displaystyle g^{u \b t} \frac{\partial g_{p \b u}}{\partial z^r} \frac{\partial g_{t\b q}}{\partial \bar z^s}$. Let us recall their expression, as computed in \cite{CGP}: 
First, we have

\begin{eqnarray*}
\frac{\partial^2 g_{p\b q}}{\partial z^r\partial z^{\b s}}(p)&= & -\frac{\delta_{pq} \, \vp_{q, r\bar s}} {(\ep^2+|z^q|^2)^{1-\beta_q}} + \nonumber \\
& + &\!\!{\alpha_{pq, \b s}\delta_{qr}+ z^q\alpha_{pq, r \b s}\over \fepq^{1-\beta_q}}-(1-\beta_q) {|z^q|^2 \alpha_{pq, r}\delta_{qs}\over \fepq^{1-\beta_q}}- \nonumber  \\ 
&- &\!\!(1- \beta_q){|z^q|^2\alpha_{pq, \b s}\delta_{qr} +\delta_{qr}\delta_{qp}\delta_{qs}- |z^q|^2\varphi_{q, r\b s}\delta_{pq}\over \fepq^{2-\beta_q}}+  \label{second_der}  
\nonumber \\
&+& (1- \beta_q)(2- \beta_q){ |z^q|^2\delta_{qp}\delta_{qr}\delta_{qs}\over \fepq^{3-\beta_q}}+ {z^{\b p}\overline{\alpha_{qp, \b rs}}+ \delta_{ps}\overline{\alpha_{qp, \b r}}\over \fepp^{1-\beta_p}}- \nonumber\\
&- & (1- \beta_p){\delta_{pr}(z^{\b p})^2\overline{\alpha_{qp, s}}+ \delta_{ps}|z^{p}|^2\overline{\alpha_{qp, \b r}}\over \fepp^{2-\beta_p}}+\nonumber \\
&+ & \sum_k{|z^{k}|^2 {\beta_{kpq, r\b s}}\over (\ep^2+ |z^k|^2)^{1- \beta_k}}+ \sum_k\beta_k{\delta_{kr}\delta_{ks}- |z^k|^2\varphi_{k, r\b s}\over (\ep^2+ |z^k|^2)^{1- \beta_k}}\theta_k-\nonumber\\
&- & \sum_k\beta_k(1-\beta_k){|z^k|^2\delta_{ks}\delta_{kr}\over (\ep^2+ |z^k|^2)^{2- \beta_k}}\theta_k+ \sum_k\beta_k{\delta_{kr}z^{\b k}\over (\ep^2+ |z^k|^2)^{1- \beta_k}}\theta_{k, \b s}+\nonumber\\
&+& \sum_k\beta_k{\delta_{ks}z^k\over (\ep^2+ |z^k|^2)^{1- \beta_k}}\theta_{k, r}+ {\cO}(1).\nonumber
\end{eqnarray*}
where we denote e.g. by $\theta_{i, k}$ the $z^k$-partial derivative of the function
$\theta_i$ (which is the $(p, \b q)$-component of the curvature form $\Theta_i$ in the chosen coordinates). \\

As for the quantity $$\frac{\partial g_{p \b u}}{\partial z^r}(p)\frac{\partial g_{t\b q}}{\partial \bar z^s}(p)$$ it is given by:
\begin{eqnarray*}
&&\delta_{pu}\delta_{ur}\delta_{tq}\delta_{qs} (\beta_p-1)(\beta_s-1) \frac{\bar z^p z^q}{\fepp^{2-\beta_p}\fepq^{2-\beta_q}}+\\
&+&\delta_{pur} \frac{z^{\bar p} z^q}{\fepp^{2-\beta_p}\fepq^{1-\beta_q}}  \times \textrm{\small (bd)} +\delta_{pur} \frac{z^{\bar p} z^t}{\fepp^{2-\beta_p}\fept^{1-\beta_t}}  \times \textrm{\small (bd)} +\\
&+&\delta_{pur} \frac{ z^{\bar p}  z^s}{\fepp^{2-\beta_p}\feps^{1-\beta_s}}  \times \textrm{\small (bd)} +\delta_{tqs}\frac{z^{ q} z^{\bar u}}{\fepq^{2-\beta_q}\fepu^{1-\beta_u}}  \times \textrm{\small (bd)}+\\
&+&\delta_{tqs}\frac{z^{ q} z^{ \bar p}}{\fepq^{2-\beta_q}\fepu^{1-\beta_p}}  \times \textrm{\small (bd)}+\delta_{tqs}\frac{z^{ q} z^{\bar r}}{\fepq^{2-\beta_q}\fepr^{1-\beta_r}}  \times \textrm{\small (bd)}+\\
&+& \frac{ z^{\bar u}  z^q}{\fepu^{1-\beta_u}\feps^{1-\beta_q}}  \times \textrm{\small (bd)} + \frac{ z^{\bar u}  z^t}{\fepu^{1-\beta_u}\fept^{1-\beta_t}}  \times \textrm{\small (bd)} +\\
&+& \frac{ z^{\bar u}  z^s}{\fepu^{1-\beta_u}\feps^{1-\beta_s}}  \times \textrm{\small (bd)} +\frac{ z^{\bar p}  z^q}{\fepp^{1-\beta_p}\fepq^{1-\beta_q}}  \times \textrm{\small (bd)} +\\
&+& \frac{ z^{\bar p}  z^t}{\fepp^{1-\beta_p}\fept^{1-\beta_t}}  \times \textrm{\small (bd)} + \frac{ z^{\bar p}  z^s}{\fepp^{1-\beta_p}\feps^{1-\beta_s}}  \times \textrm{\small (bd)} +\\
&+& \frac{ z^{\bar r}  z^q}{\fepr^{1-\beta_r}\fepq^{1-\beta_q}}  \times \textrm{\small (bd)} + \frac{ z^{\bar r}  z^t}{\fepr^{1-\beta_r}\fept^{1-\beta_t}}  \times \textrm{\small (bd)} +\\
&+& \frac{ z^{\bar r}  z^s}{\fepr^{1-\beta_r}\feps^{1-\beta_s}}  \times \textrm{\small (bd)} +\frac{\partial g_{p \b u}}{\partial z^r}(p) \times \textrm{\small (bd)} + \frac{\partial g_{t\b q}}{\partial \bar z^s}(p) \times \textrm{\small (bd)}.\\
\end{eqnarray*}
where {\small{(bd)}} means "bounded terms".\\

\medskip

\subsection{Lower bounds on the curvature}
\label{subsec:lower}
\noindent Let now $u= \sum u_p \frac{\d}{\d z_p}, v=\sum v_r  \frac{\d}{\d z_r}$ be vector fields with norm $1$: $|u|^2_{\ome} = |v|^2_{\ome} =1$. Using the expressions above and the observation \eqref{obs}, we conclude the existence of a constant $C> 0$ so that the following facts hold true.

\begin{enumerate}

\item If all the indexes $(p, q, r, s)$ are distinct, then we have
$$\left| R_{\varepsilon p\ol q r \ol s} \right| |u_p \bar u _q v_r \bar v_s| \leq C.$$

\smallskip

\item If only two indexes among $(p, q, r, s)$ are identical, we have 
$$\left| R_{\varepsilon p\ol p r \ol s} \right| |u_p \bar u _p v_r \bar v_s| \leq C.$$
as well as
\begin{equation*}
 R_{\varepsilon p\ol q p \ol s}u_p \bar u _q v_p \bar v_s + 
R_{\varepsilon q\ol p s \ol p}u_q \bar u _p v_s \bar v_p   \geq -\frac{C}{(\varepsilon^2+ |z^p|^2)^{1/2}}
|u_p| |v_p|- C \\
\end{equation*}
if $p\in\{1,..., d\}$. If $p\in\{d+1,..., n\}$, then we have
\begin{equation*}
R_{\varepsilon p\ol q p \ol s}u_p \bar u _q v_p \bar v_s + 
R_{\varepsilon q\ol p s \ol p}u_q \bar u _p v_s \bar v_p  \geq -C
\end{equation*}

\smallskip

\item If the collection of indexes $(p, q, r, s)$ consists of two distinct elements, then we have:

$$  R_{\varepsilon p\ol p r \ol r} |u_p \bar u_p v_r \bar v_r|\geq -C$$
as well as

\begin{equation*}
 R_{\varepsilon p\ol p p \ol s} u_p \bar u_p v_p \bar v_s+ 
R_{\varepsilon p\ol p s \ol p} u_p \bar u_p v_s \bar v_p
\geq \frac{-C}{(\ep^2+|z^p|^2)^{1/2}}  |u_p|^2 |v_p|- C
\end{equation*}
if $p\in \{1, \ldots, d\}$ and 
\begin{equation*}
 R_{\varepsilon p\ol p p \ol s} u_p \bar u_p v_p \bar v_s+ 
R_{\varepsilon p\ol p s \ol p} u_p \bar u_p v_s \bar v_p
\geq- C
\end{equation*}
if $p\in \{d+1, \ldots, n\}$.\\

We also have
\begin{equation*}
R_{\varepsilon p\ol q p \ol q}  u_p \bar u_q v_p \bar v_q + 
R_{\varepsilon q\ol p q \ol p} u_q \bar u_p v_q \bar v_p \geq \frac{C}{(\varepsilon^2+ |z^p|^2)^{1/2}(\varepsilon^2+ |z^q|^2)^{1/2}}
| u_p \bar u_q v_p \bar v_q|- C 
\end{equation*}
if $p, q\in \{1,..., d\}$. If $p\in \{1,..., d\}$ and $q\in \{d+1,..., n\}$, then we have

\begin{equation*}
R_{\varepsilon p\ol q p \ol q}  u_p \bar u_q v_p \bar v_q + 
R_{\varepsilon q\ol p q \ol p} u_q \bar u_p v_q \bar v_p \geq \frac{C}{(\varepsilon^2+ |z^p|^2)^{1/2}}
| u_p| |v_p|- C 
\end{equation*}
and if $p, q\in \{d+1,..., n\}$, then we have
\begin{equation*}
R_{\varepsilon p\ol q p \ol q}  u_p \bar u_q v_p \bar v_q + 
R_{\varepsilon q\ol p q \ol p} u_q \bar u_p v_q \bar v_p \geq- C 
\end{equation*}
The inequalities corresponding to the 
other curvature coefficients are deduced from the previous ones, by symmetry.

\smallskip
\item In the remaining cases we have
\begin{eqnarray}
\nonumber
R_{\varepsilon p\ol p p \ol p}|u_p|^2 |v_p|^2 \geq -\frac{C}{\varepsilon^2+ |z^p|^2}
|u_p|^2 |v_p|^2- C,
\nonumber
\end{eqnarray}
if $p\in \{1,..., d\}$, and 
\begin{eqnarray}
\nonumber
R_{\varepsilon p\ol p p \ol p} |u_p|^2 |v_p|^2 \geq -C
\nonumber
\end{eqnarray}
if $p\in \{d+1,..., n\}$.\\
\end{enumerate}

\begin{rema}  Given the expressions above and \eqref{obs}, it is clear that the curvature tensor will be bounded from below by a constant independent of $\varepsilon$ if the coefficients
$\beta_j$ belong to the interval $]0, 1/2]$. However, simple examples show that the estimates above are practically optimal, i.e. in general we do not have this property.
\end{rema}

%%%%%%%%%%%%%%%%%%%%%%%%%%%%%%%%%%%%%%%%%%%%%%%%%%%%%%%%%%%%%%%%%%%%%%%%%%%%%%%%%%%%

\section{A useful auxiliary function}
\label{sec:aux}

\noindent Let $\rho\in (0, 1)$ be a real number; we consider the function
$$\Psi_{\ep, \rho}:= C\sum_{k= 1}^d\chi_\rho(\ep^2+ |s_k|^2) ;$$
we remark that functions of this kind have already appeared in he construction of the metric $\omega_\varepsilon$ in \S \ref{sec:reg2}. By the relation \eqref{5}, we have 
\begin{equation}
Ci\ddbar \big( \chi_\rho(\varepsilon^2+ |s_k|^2)\big)\geq (\varepsilon^2+ |s_k|^2)^{\rho-1}\sqrt{-1}\langle D^\prime s_k, D^\prime s_k\rangle- {C\over \rho}\omega_\varepsilon;\\
\end{equation}
hence by adjusting the constants, we obtain
\begin{equation}
\label{eq:psi2}
i\ddbar \Psi_{\ep, \rho}\geq C\sum_{k=1}^d
(\varepsilon^2+ |s_k|^2)^{\rho-1}\sqrt{-1}\langle D^\prime s_k, D^\prime s_k\rangle- {C}\omega_\varepsilon.
\\
\end{equation}
In the above equation, we do not write the parameter $\rho$ because it will be fixed
at the end of the proof (and of course, the two constants $C$ above are not the same, we hope that this is not too confusing). Similarly, we will use the lighter notation $\Psi_{\ep}$ instead of  $\Psi_{\ep, \rho}$. \\

We are willing to use these functions $\Psi_{\ep}$ in the Proposition \ref{prop:key} where they would play the role of $\Psi$. We already saw in \S \ref{sec:reg2} that these functions are uniformly bounded in $\ep$, and 
it follows from \eqref{eq:psi} that they all are $C\ome$-psh for some fixed $C>0$. Therefore condition $(ii)$ of Proposition \ref{prop:key} is fulfilled. The only remaining property to check concerns the curvature of $\ome$; namely we want to see that the inequality 
\begin{equation}
\label{eq:curv}
i\Theta_{\ome}(T_X) \ge -(C\ome+ \ddc \Psi_{\ep})\otimes \mathrm{Id}
\end{equation} 
is satisfied. \\

To see it, we use the coordinate system $(z)$ introduced in the previous section, and adapted to $(X, D)$. Then \eqref{eq:curv} is equivalent to showing that for any vector fields $u,v$, we have:
\[R_{p\bar q r \bar s} u_p \bar u_q v_r \bar v_s \ge - C |u|_{\ome}^2 |v|_{\ome}^2 - \Psi_{\ep, p\bar q} u_p \bar u_q |v|_{\ome}^2\]
Without loss of generality, one can assume that $u,v$ are normalized, and the previous inequality reduces to 
\[R_{p\bar q r \bar s} u_p \bar u_q v_r \bar v_s \ge - C - \Psi_{\ep, p\bar q} u_p \bar u_q\]
Using \eqref{eq:psi2}, it would be enough to show that 
\begin{equation}
\label{eq:curv2}
R_{p\bar q r \bar s} u_p \bar u_q v_r \bar v_s \ge - C -C \sum_{p=1}^d \frac{1}{(\ep^2+|z_p|^2)^{1-\rho}} |u_p|^2
\end{equation}

To show that the previous inequality is satisfied, we separate the curvature terms according to how many distinct indexes appear in $(p,q,r,s)$, and use the estimates established in \S \ref{subsec:lower}.\\

\begin{itemize}
\item[$(i)$] If the indexes $(p, q, r, s)$ are distinct, then the relation (1) shows that we have  
\[R_{p\bar q r \bar s} u_p \bar u_q v_r \bar v_s \ge - C\]
(so that the sum above runs over all $(p, q, r, s)$ which are distinct).\\

\item[$(ii)$] If only two indexes among $(p, q, r, s)$ are identical, by the relation (2), we only have to show that
\[\frac{1}{(\varepsilon^2+ |z^p|^2)^{1/2}} |u_p| |v_p| \le  \frac{C}{(\ep^2+|z_p|^2)^{1-\rho}} |u_p|^2+C\]
But using \eqref{obs} and the basic inequality $x \le x^2+1$, the above relation will hold true as soon as $\rho>1-\beta_p$ for each $p$.\\

\item[$(iii)$] If the collection of indexes $(p, q, r, s)$ consists of two distinct elements, then by the relation $(3)$, there are three different inequalities to check. The first one is:
\[ \frac{1}{(\ep^2+|z^p|^2)^{1/2}}  |u_p|^2 |v_p| \le  \frac{C}{(\ep^2+|z_p|^2)^{1-\rho}} |u_p|^2+C\]
 As $\eqref{obs}$ yields $ \frac{|v_p|}{(\ep^2+|z^p|^2)^{1/2}} \le \frac C {(\ep^2+|z^p|^2)^{\beta_p/2}}$, the above inequality will hold as soon as $\rho>1-\beta_p/2$ for each $p$. 
 
The second inequality involves two indexes $p,q \in \{1, \ldots, d\}$:
\[ \frac{ | u_p \bar u_q v_p \bar v_q| }{(\ep^2+|z^p|^2)^{1/2}(\ep^2+|z^q|^2)^{1/2}}\le  \frac{C}{(\ep^2+|z_p|^2)^{1-\rho}} |u_p|^2+\frac{C}{(\ep^2+|z_q|^2)^{1-\rho}} |u_q|^2+C\]
and can be reduced (using the arithmetic-geometric mean inequality) to the following inequality:
\[ \frac{1}{\ep^2+|z^p|^2}|u_p|^2 |v_p|^2 \le  \frac{C}{(\ep^2+|z_p|^2)^{1-\rho}} |u_p|^2+C\]
which follows from \eqref{obs} as soon as $\rho>1-\beta_p$ for each $p$.

The third and last inequality to check is
\[ \frac{1}{(\ep^2+|z^p|^2)^{1/2}}|u_p| |v_p| \le  \frac{C}{(\ep^2+|z_p|^2)^{1-\rho}} |u_p|^2+C\]
and has already been treated in $(ii)$. 

\item[$(iv)$] The last case is when $p=q=r=s$, where we need to make sure that 
\[\frac{1}{\varepsilon^2+ |z^p|^2}
|u_p|^2 |v_p|^2 \le \frac{C}{(\ep^2+|z_p|^2)^{1-\rho}} |u_p|^2+C \]
which we checked in $(iii)$ already.\\
\end{itemize}

To conclude, if $\rho$ is chosen small enough, the functions $\Psi_{\ep}$ introduced in the beginning of this sections satisfy \eqref{eq:curv2} and hence \eqref{eq:curv}.

\section{Proof of Theorem A}

In this section, we gather the arguments and computations of the previous sections to complete the proof of Theorem A. 

Let us recall the notations: $(X,D)$ is a log smooth klt pair, $\om$ is a given Kähler metric, $dV$ is a smooth volume form. As for $D=\sum (1-\beta_k) Y_j$, its components $Y_k$ are smooth hypersurfaces cut out by sections $s_k$, and the coefficients $\beta_k$ belong to $(0,1)$. Finally, we are given a (bounded) solution $\vp\in \mathrm{PSH}(X, \om)$ of the following equation (we assume $dV$ to normalized if $\mu=0$): 
\begin{equation}
\label{eq1}
(\om+\ddc \vp)^n =  \frac{e^{\mu \vp} dV}{\prod_{k=1}^r |s_k|^{2(1-\beta_k)}}
\end{equation}
and we want to show that $\omvp:= \om + \ddc \vp$ has cone singularities along $D$. The strategy is to regularize the Monge-Ampère equation, but in the case $\mu<0$ (which should be thought as the "positive curvature" case), we have to be more cautious and treat the regularization differently.

\subsection{The $L^{\infty}$ estimate}
\label{sec:reg}
Here we explain how to derive the zero-order estimate for the potential solution of the regularized equation $(\mathrm{MA}_{\ep})$ considered in the introduction.
\subsubsection*{The case $\mu \ge 0$}
In that case, we simply consider the equation
\begin{equation}
\label{eq20}
(\om+\ddc \vpe)^n =  \frac{e^{\mu \vpe} dV}{\prod_{k=1}^r (|s_k|^{2}+\ep^2)^{(1-\beta_k)}}
\end{equation}

We can proceed exactly as in \cite[\S 5.1]{CGP} to obtain uniform $\mathcal C^0$ estimates for $\vpe$: $||\vpe||_{\infty} \le C$. In one word, we use Ko\l odziej's estimates in the case $\mu=0$, and if $\mu>0$, we use the approximate cone metric and the maximum principle. Moreover, the argument at the end of $\S 5$ in \cite{CGP} shows that $\vpe$ converges to $\vp$ in $L^1$ (every limit of a subsequence of $\vpe$ converge to $\vp$ by uniqueness of the solution of the MA equation).

\subsubsection*{The case $\mu < 0$}
In this case, even if \eqref{eq1} has a solution by assumption, we cannot guarantee that $\eqref{eq20}$ will also have one for $\ep>0$ small enough. So we begin by approximating $\vp$  with a decreasing sequence of smooth $\om$-psh functions $\phi_{\ep}$ satisfying 
\begin{equation}
\label{eq:psi}
\ddc \phi_{\ep} \ge -C \om
\end{equation} for some uniform $C>0$. This is possible thanks to Demailly's regularization theorem \cite{D1, D2}. In particular the $\phi_{\ep}$'s are uniformly bounded:
\begin{equation}
\label{eq:l3}
\sup_X |\phi_{\ep}|\le C
\end{equation}
for some $C$ independent of $\ep$.
Then, we consider the following equation (in $\vpe$):

\begin{equation}
\label{eq3}
(\om+\ddc \vpe)^n =  \frac{e^{\mu \phi_{\ep}} dV}{\prod_{k=1}^r (|s_k|^{2}+\ep^2)^{(1-\beta_k)}}
\end{equation}

\vspace{3mm}
\noindent
By multiplying $dV$ with a constant, we can make sure that the total mass of the RHS is $\{\om\}^n$; this constant depends on $\ep$, but in a totally harmless way because $\frac{e^{\mu \phi_{\ep}}}{\prod_{k=1}^r (|s_k|^{2}+\ep^2)^{(1-\beta_k)}}$ is uniformly bounded in $L^1(dV)$ by the klt condition $\beta_k<1$. Therefore we can assume that the volume form is already normalized, and using Ko\l odziej's estimates (the rhs is uniformly in $L^p$ for some $p>1$), we get 
\begin{equation}
\label{eq:c0}
\sup_X |\vpe| \le C
\end{equation}

for some uniform $C>0$. Finally, we know that $\vpe$ converges to $\vp$ in $L^1$ thanks to Ko\l odziej's stability theorem for instance, so that the Laplacian estimates that we will obtain in the next section for $\vpe$ will actually yield Laplacian estimates for the initial solution $\vp$.\\

Note that we could equally have used this regularization process (the introduction of $\phe$) to deal with the case where $\mu \ge 0$, instead of using the maximum principle. Besides, this argument will work in the general klt case whereas the maximum principle will not. 

\subsubsection*{Conclusion}
So in both cases, we produced smooth metrics $\om+\ddc \vpe$ with uniformly bounded potentials converging to $\om+\ddc \vp$ in the topology of currents. Therefore, in order to prove Theorem A, it will be enough to show that $\om+\ddc \vpe$  is uniformly equivalent to some approximation of the cone metric, e.g. the Kähler metric $\ome$ constructed in \ref{sec:reg2}.

\subsection{The Laplacian estimate}
As we saw above, we are interested in the following family of 
Monge-Amp\`ere equations: 
\[(\om+\iddb \vpe)^n = \frac{e^{\mu \phi_{\ep}}dV} {\prod_{k=1}^d(\ep^2+|s_k|^2)^{1-\dej}} \]
where $\phi_{\ep}$ is a smooth approximation of the solution $\vp$ of \eqref{eq1}. We rewrite the last equation under the following form:
\[(\om+\iddb \vpe)^n = e^{\fe+\mu\phi_{\ep} } \ome^n\]
where \[\fe=-\log \left(\frac  {\prod_{k=1}^d(\ep^2+|s_k|^2)^{1-\dej} \ome^n}{dV} \right). \]
\smallskip
We want to show the following Laplacian estimate
\begin{equation}
\label{lap}
C^{-1}\ome \le \omvpe \le C \ome
\end{equation} 
and to achieve this goal, we will use Proposition \ref{prop:key} with background metric $\ome$, $\Psi=\Psi_{\ep}$, and $(\pp, \psm)=(\fe+\mu \phe, 0)$ if $\mu \ge 0$ and $(\pp, \psm)=(\fe,-\mu \phe)$ if $\mu< 0$. \\

There are five conditions that need to be fulfilled if one wants to apply that proposition. The first condition is a uniform bound on $\sup |\vpe|$ which has already been obtained in \S \ref{sec:reg};  conditions
$(ii)$ and $(iv)$ concerning $\Psi_{\ep}$ and the curvature of $\ome$ have been checked in  \S \ref{sec:aux}; condition $(iv)$ is a direct consequence of \eqref{eq:psi} and \eqref{eq:l3}. The last thing to check is condition $(iii)$ concerning $\fe$; the bound $\sup |\fe|$ is easy and proved in \cite{CGP}, but the inequality
\[\ddc \fe \ge -(C\ome + \ddc \Psi_{\ep})\]
is more difficult and needs some further computations.\\

To obtain this inequality, we are going to use the computations of \cite[\S 4.5]{CGP}. These show that $\ddc \fe$ is bounded (up to some universal multiplicative constant) by sums (indexed on $i$) of terms like
\[\frac{dz^p \wedge d \bar z^p}{(\ep^2+|z^p|^2)^{\alpha_p}} \quad \textrm{or} \quad \frac{d z^p\wedge d\bar z^q+d \bar z^p\wedge d z^q}{(\ep^2+|z^p|^2)^{\alpha_p'}(\ep^2+|z^q|^2)^{\alpha_q'}}  \]
where $\alpha_p \in \left\{\beta_p, 1-\beta_p, \frac {1-\beta_p} {2}\right\}$ and $\alpha_p' \in \left\{\frac 1 2 - \beta_p, \beta_p-\frac 1 2 \right\}$. We deduce from this that $\pm \ddc \fe$ is dominated by
\[\sum_{i,p}   \frac{dz^p \wedge d \bar z^p}{(\ep^2+|z^p|^2)^{\tilde \beta_p}} \]
where $\tilde \beta_p := \max \{\beta_p, 1-\beta_p\}$. But from \eqref{eq:psi2}, it is clear that $C\ome+ \ddc \Psi_{\ep} \ge \ddc \fe$ for $C$ big enough, and $\rho < \min_p \min\{\beta_p, 1-\beta_p\}$. \\

\subsubsection*{Conclusion} Finally,  we may apply Proposition \ref{prop:key} in our setting which will give us the expected Laplacian estimate \eqref{lap}. As we observed in \S \ref{sec:reg} that $\omvpe$ was converging as a current to the given solution $\omvp$ of our initial Monge-Ampère equation \eqref{eq1}, the arguments given in the introduction apply to get the smooth convergence of $\omvpe$ to $\omvp$ on the compact subsets of $X\setminus \Supp(D)$, and the estimate \eqref{lap} guarantees that $\omvp$ will have cone singularities along $D$. Theorem A is thus proved.

\begin{rema}
\label{rem:unif}
As a consequence of all these computations, it happens that the Laplacian estimate will hold uniformly when the angles $\beta_k$ vary in a fixed range $[\delta, 1-\delta]$ for some $\delta>0$. This remark will play an important role in the proof of Theorem B.
\end{rema}

\section{Applications and generalizations}

\subsection{Vanishing of holomorphic tensor fields}

Our first application is a consequence of \cite{CGP} concerning the vanishing/parallelism of holomorphic tensors as defined by Campana. Recall that $T_s^r(X|D)$ can be defined as the vector bundle on $X$ whose sections are tensors of $T_s^r X := (\bigotimes^r T_X) \otimes (\bigotimes^s T_X^*)$ over $X\setminus \Supp(D)$ which are \textit{bounded} with respect to some metric with conic singularities along $D$ (cf \cite{Camp1} or \cite{CGP} for a more algebraic definition). As a consequence of the Corollary stated in the introduction and the techniques in \cite{CGP} we get the following:

\begin{theo}
\label{thm:van}
Let $(X, D)$ be a log smooth klt pair. Then the following assertions hold true.
\begin{enumerate}
\item[$(i)$] If $c_1(K_X+ D)$ contains a K\"ahler metric, then we have
$H^0\big(X, T_s^r(X|D)\big)= 0$
for any $r\ge s+ 1$.
\item[$(ii)$] If $-c_1(K_X+ D)$ contains a K\"ahler metric, then we have
$H^0\big(X, T_s(X|D)\big)= 0$
for any $s\ge 1$.
\item[$(iii)$] If $c_1(K_X+ D)$ contains a smooth, semi-positive 
(resp. semi-negative) representative, then the holomorphic sections of the bundle 
$T^r(X|D)$ (resp. $T_s(X|D)$) are parallel.\\
\end{enumerate}
\end{theo}

\subsection{The general klt case}

Our second application is a generalization of the main result of \cite{G2}, where we deal with general klt pairs, not necessarily log smooth ($X$ may be singular, and $D$ does not necessarily have normal crossing support):

\begin{theo}
\label{thm:klt}
Let $(X,D)$ be a klt pair, and let $\mathrm{LS}(X,D):= \{x\in X; (X,D)$ is log smooth at $x\}$.
\begin{enumerate}
\item[$(i)$] If $K_X+D$ is big, then the Kähler-Einstein metric of $(X,D)$ has cone singularities along $D$ on $\mathrm{LS}(X,D)\cap \mathrm{Amp}(K_X+D)$. 
\item[$(ii)$] If $-(K_X+D)$ is ample, then any Kähler-Einstein metric for $(X,D)$ has cone singularities along $D$ on $\mathrm{LS}(X,D)$.\\
\end{enumerate}
\end{theo}

This result is proved in \cite{G2} under the assumption that the coefficients of $D$ belong to $[1/2, 1)$. 
We do not intend to give a detailed proof of this result but we can still outline the main steps of the proof, which is a combination of the techniques in \cite{G2} and the new input of this paper. 

\begin{proof}[Sketch of proof]
First, we take a log resolution of the pair (we also have to reduce the big case to the ample case by considering a log canonical model whose existence is guaranteed by \cite{BCHM}). We get a Monge-Ampère equation of the following type: 
\[(\pi^*\om + \ddc \vp)^n = \prod |t_j|^{2a_j} \frac{e^{\mu \vp}dV}{\prod |s_i|^{2(1-\beta_i)}}\]
where $a_j >-1$ for all $j$ (klt condition) and $(s_i=0)$ is the equation of the strict transform of the $i$-th component of $D$ under the log resolution $\pi : X' \to X$; finally $\om$ is a Kähler form on $X$, and $dV$ is a smooth volume form on $X'$. 

To analyze the regularity of the solution $\vp$ of this equation, we regularize the equation in the following way: 
\[(\pi^*\om + \ddc \vpe)^n = \prod |t_j|^{2a_j} \frac{e^{\mu \phe}dV}{\prod (|s_i|^2+\ep^2)^{(1-\beta_i)}}\]
where $\phe$ is as in \S \ref{sec:reg} a regularization of $\vp$. To construct the approximate cone metric with good curvature properties, 
we first add $\ddc \log |s_E|^2$ to $\pi^*\om$ for some exceptional $\Q$-divisor $E$: we get a smooth Kähler form $\om_E$ on $X'\setminus \Supp(E)$ that extends to $X'$ 
as a Kähler form. Then we add $\ddc \pse$ (cf \S \ref{sec:reg2}) to $\om_E$ to get a Kähler form $\ome$ on $X'$ that behaves exactly as our approximate cone metric used in the previous sections. 

Although $\ome$ does not live in the same cohomology class as the unknown metric $\omvpe:= \pi^*\om + \ddc \vpe$, one can still use the usual Laplacian estimate on $X'\setminus E$. And adding under the Laplacian the term $\Psi_{\ep}$ considered in this paper will enable us to get rid of the unbounded (from below) curvature terms. Then one can compare on this Zariski open subset $\omvpe$ and $\ome$. Actually there is a new difficulty coming from the term $\prod |t_j|^{2a_j}$ (and particularly when some $a_j$ are negative) appearing in the rhs of the Monge-Ampère equation. But we can overcome it using the general estimate appearing in \cite{Paun} and expanded in \cite[Theorem 10.1]{BBEGZ}: for this precise point, there is no difference with \cite{G2} once the curvature issues of $\ome$ have been dealt with. 

As for the rest, it is relatively classic and completely identical to \cite{G2}, to which we refer the reader looking for details. 
\end{proof}

\subsection{Mixed cone and cusp singularities}

The last generalization concerns log smooth \textit{log canonical} pairs, ie pairs $(X,D)$ consisting in a compact Kähler manifold $X$ and a divisor $D= \sum_j (1-\beta_j)D_j + \sum_k D_k$ having simple normal crossing support and coefficients $\beta_j \in (0,1)$. We write $X_0:=X \setminus \Supp (D)$, $D= \Dklt+\Dlc$ where $\Dklt:= \sum_j (1-\beta_j)D_j$ and $\Dlc=\sum_k D_k$, and we denote by $J$ (resp. $K$) the set of indexes for the klt (resp. lc) part. \\

Whenever $\Dklt=0$ (the "absolute case"), it was showed by Kobayashi \cite{Koba} and Tian-Yau \cite{TY} that whenever $K_X+D$ is ample, there exists a unique negatively curved Kähler-Einstein metric on $X_0$ having cusp (also called Poincaré) singularities along $D$. Actually, in \cite{TY} the result is extended to the case where $\Dklt$ is an orbifold divisor, i.e. when $\beta_j = \frac 1 {m_j}$ for some integers $m_j$; the Kähler-Einstein metric being then an orbifold metric along $\Dklt$. \\

More generally, one can look for Kähler-Einstein metrics (with negative curvature) attached to the pair $(X,D)$, those metrics being usual Kähler-Einstein metrics on $X_0$, and having mixed cone and cusp singularities along $D$, ie being locally quasi-isometric to the model
\[\om_{\rm mod}:=\sum_{j=1}^r \frac{i dz_j\wedge d\bar z_j}{|z_j|^{2a_j}} + \sum_{j=r+1}^s \frac{i dz_j\wedge d\bar z_j}{|z_j|^{2} \log^{2}|z_j|^{2}}+\sum_{j=r+s+1}^n i dz_j\wedge d\bar z_j \] 
if $(X,D)$ is locally isomorphic to $(X_{\rm mod}, D_{\rm mod})$, where $X_{\rm mod}=(\mathbb{D}^*)^r\times (\mathbb{D}^*)^s \times \mathbb{D}^{n-(s+r)}$, $D_{\rm mod}=a_1 [z_1=0]+\cdots+a_r [z_r=0]+[z_{r+1}=0]+\cdots + [z_{r+s}=0]$; $\mathbb{D}$ (resp. $\mathbb{D}^*$) being the disc (resp. punctured disc) of radius $1/2$ in $\mathbb C$.

This question was studied in \cite{G12} using a pluripotentialist approach, where existence and uniqueness of those metrics was proved under the assumption that $\beta_j \in (0, 1/2]$ for all $j$.
Capitalizing on the techniques developped in this paper, we are now in position to prove the general statement: 

\begin{theo}
%\phantomsection
\label{thm:lc}
Let $(X,D)$ be a log smooth log canonical pair such that $K_X+D$ is ample. Then there exists a unique Kähler metric $\om$ on $X_0$ such that 
\begin{enumerate}
\item[$(i)$] $\Ric \om = -\om$ on $X_0$
\item [$(ii)$] $\om$ has mixed cone and cusp singularities along $D$.
\end{enumerate}
\end{theo}

\begin{proof}[Sketch of proof]
As in the previous section, we do not give a detailed proof of the Theorem since the main arguments are already in \cite{G12}.

\noindent
The first step is to express the problem in terms of Monge-Ampère equations. Picking a Kähler form $\om_0 \in c_1(K_X+D)$, and sections $s_j$ (resp. $s_k$) cutting out $D_j$ (resp. $D_k$), we are led to show that the (unique) solution of the following Monge-Ampère equation:
\[(\om_0+\ddc \vp)^n = \frac{e^{\vp}dV}{\prod_j |s_j|^{2(1-\beta_j)} \cdotp \prod_k |s_k|^2}\]
is smooth on $X_0$ and has mixed cone and cusp singularities along $D$; here $dV$ is a smooth volume form, $\vp \in \mathcal E(X, \om)$ and the product $(\om_0+\ddc \vp)^n$ is understood in the non-pluripolar Monge-Ampère product sense, cf \cite{GZ07}. 

The uniqueness is already proved in \cite{G12} using the comparison principle of \cite{GZ07}. As for the regularity, we introduce a parameter $\ep>0$ and reformulate the equation as
\[(\ome+\ddc \vpe)^n = e^{\vpe+\fe}\ome^n\]
where $\ome = \om_0+ \ddc \pse + \sum_{k\in K} \ddc\left(-\log \log^2 |s_k|^2\right)$  ($\pse$ being defined as in \eqref{pse}) and \[\fe=\pse-\log \left(\frac{\prod_j (|s_j|^{2}+\ep^2)^{1-\beta_j}\cdotp \prod_k |s_k|^2 \log^2 |s_k|^2 \cdotp \ome^n}{dV}\right)\]
The metric $\ome$ has cusp singularities along $\Dlc$ and is a "smooth" approximation of a cone metric along $\Dklt$. If we prove that $\omvpe:=\ome+\ddc \vpe$ is uniformly quasi-isometric to $\ome$, then passing to the limit and using \cite[Proposition 2.5]{G12}, this will yield the expected result. We are going to use our Laplacian estimate \ref{prop:key} or more precisely its generalization to complete Kähler manifold with Ricci curvature bounded from below (in a non quantitative way). This generalization is obtained using Yau's maximum principle instead of the classical maximum principle, cf \cite[Proposition 1.10]{G12}. Therefore, we have to show three things: \[(i)\, \sup |\vpe|\le C, \,(ii) \, \ddc \fe \ge -(C\ome+ \ddc \Pse), \,  (iii) \,  \Theta_{\ome}(T_X) \ge -(C\ome+ \ddc \Pse) \otimes \mathrm{Id}.\] 
The $L^{\infty}$ bound $(i)$ is easily obtained using Yau's maximum principle; condition $(ii)$ can be checked exactly as in this paper using the computations of \cite[ \S 4.2.3]{G12}. As for $(iii)$, the calculations of \cite[\S 4.2.2]{G12} show that the cone and cusp parts do almost not interfere in the curvature terms in the sense that the singularities of the components $R_{p\bar q r \bar s}$ do not get worse that the one of the cone metric. In particular condition $(iii)$ is fulfilled, which concludes the proof.
\end{proof}

With this result at hand, we can generalize the vanishing theorem for holomorphic orbifold tensors given in \cite{G12} for log smooth log canonical pairs with general coefficients in $[0,1]$: more precisely, if $(X,D)$ is a log smooth log canonical pair with $K_X+D$ ample, then there is non-zero holomorphic tensor (in the sense of Campana) of type $(r,s)$ whenever $r\ge s+1$.

\section{Higher regularity}

In this section we prove Theorem B, stating that the second conic derivatives of bounded solutions of (MA) are Hölder continuous with respect to the conic metric. 
Our starting observation is that it is sufficient to establish the result assuming that the coefficients 
$\beta_k$ are rational numbers, as long as the estimates we obtain in this case 
are uniform enough so that the general case can be derived by a limiting process
(so that we are implicitly using the stability results for (MA) equation).

As we can see from \cite{Siu}, in the classical setting $D=0$ the H\"older estimates 
of the second derivatives of the solution of a Monge-Amp\` ere equation are obtained 
thanks to the fact that the Hessian of $\varphi$ evaluated in the direction of a vector $v$ 
\emph{with constant coefficients} satisfies a certain differential inequality. 
In our case, the vector $v$ we have to consider should correspond to a local section of the 
\emph{orbifold tangent bundle}; 
it is naturally multi-valued, e.g.  
$$\displaystyle v=\sum_j a_j{z_j}^{1-\beta_j}\frac{\partial}{\partial z_j},$$ 
where $a_j$ are constants.
We cannot work directly with such object; however, 
the rationality assumption will allow us to pull back our data on a branched cover, chosen so that the above ``vector field" becomes single-valued, albeit meromorphic; we will refer to it as  
\emph{twisted vector field}. 
The estimate we have established in Theorem A is equivalent with the fact that the Hessian type operator defined by a basis of twisted vector fields gives a metric which
is quasi-isometric with the Euclidean metric when applied to the pull-back of the potential of the solution of (MA). However, the said metric is not K\"ahler, and this is an important source of difficulties.

%so that the reference conic metric becomes smooth with respect to the natural twisted differential operators. 
Nevertheless, we are able to establish a Harnack inequality for the cone metric (more precisely for its pull-back by the cover), and Theorem B follows. 
We deduce the case of general real coefficients by perturbation. 

We have divided this last section of our paper into three parts: first we recall some definitions and properties related to the cone geometry (Donaldson's spaces, Sobolev inequality, branched covers, etc.). Then we prove the conic version of Harnack inequality for both geodesic balls and balls centered on $\Delta$, and finally we run Evans-Krylov's argument to conclude.

\subsection{Conic Hölder spaces}
\label{sec:func}
\subsubsection{Donaldson's spaces}
\label{sec:don}
Let us recall the setup. We fix $\D^n \subset \CC^n$ the unit polydisk centered at the origin, and a divisor $D= \sum_{k=1}^d (1-\beta_k) D_k$ where $D_k=(z_k=0)$, $\beta_k \in (0,1)$ for all $k$, and $d\le n$. We set $\Delta:= \Supp(D)$. We denote by $\omc$ the standard cone metric attached with $(\CC^n, D)$, i.e.
\[\omc:= i\sum_{k=1}^d \frac{ dz_k \wedge d \bar z_k}{|z_k|^{2(1-\beta_k)}} +i\sum_{k=d+1}^n dz_k \wedge d \bar z_k  \]
It defines on $\CC^n\setminus \Delta$ a Riemannian metric $g_{\beta}$, whose induced distance will be denoted by $d_{\beta}$. Following \cite{Don}, one defines, for $f$ a locally integrable function on $\D^n$ and $\alpha \in (0,1)$:
\[ ||f||_{\alpha, \beta}:= \sup_{\D^n} |f|+\sup_{x\neq y}\frac{|f(x)-f(y)|}{d_{\beta}(x,y)^{\alpha}}  \]
We will write: \[\mathscr C^{\alpha, \beta}_{\rm Don}:=\{f\in L^{\infty}(\D^n);  ||f||_{\alpha, \beta} <+\infty \}\]

For the higher derivatives, one would like to use the vector fields $z_k^{1-\beta_k} \frac{\d}{\d z_k}$ but these are not well defined and choosing a local branch would not give a reasonable definition. To overcome this difficulty, one could either "work with absolute values" as in \cite{Don} or use ramified covers to mimic the orbifold case. We will now present the two approaches and compare them. \\

\noindent
Let us set, for $k=1 \ldots n$,  $\xi_k= |z_k|^{1-\beta_k} \frac{\d}{\d z_k}$ if $1\le k \le  d$, and $\xi_k=\frac{\d}{\d z_k}$ else. 

\noindent
A $(1,0)$-form $\tau$ is said to be of class $\mathscr C^{\alpha, \beta}$ in the sense of Donaldson (or for short in $\mathscr C^{\alpha, \beta}_{\rm Don}$) if for all $k$, $\tau(\xi_k) \in \cab_{\rm Don}$.
Moreover, a $(1,1)$-form $\sigma$ is said to be of class $\mathscr C^{\alpha, \beta}$ in the sense of Donaldson (or for short in $\mathscr C^{\alpha, \beta}_{\rm Don}$) if for all $k,l$, we have $\sigma(\xi_k, \xi_l) \in \cab_{\rm Don}$. Finally, we let: 
 \[\mathscr C^{2,\alpha, \beta}_{\rm Don}:=\{f\in L^{\infty}(\D^n);  f,\d f, \d\bar \d f \in \mathscr C^{\alpha, \beta}_{\rm Don} \}\]
and define the associated norms in the obvious way.\\
 
\noindent
The typical example of such functions is $|z_k|^{2\beta_k}$ which is essentially the potential of the model cone metric on $\CC^\star$. This definition has the advantage to be simple and to provide function spaces that are handy to manipulate in view of the continuity method (cf \cite{Don}). However, because of the absolute value in the definition of $\xi_k$, a function like $z^{\beta}$ (more precisely a local branch of it) will not belong to this space whereas this is supposed to be the appropriate coordinate function. To incorporate that kind of functions, we can mimic the definition in the orbifold case (when the angles can be written as $\beta_k=1/m_k$ for an integer $k$) using ramified cover. This is the object of the following section.
  
\subsubsection{Local branched covers}
\label{sec:tw}

In this section, we will always assume that the coefficients $\beta_k\in (0,1)$ are rational numbers, and we will write them in their irreducible form $\beta_k=p_k/q_k$ with $p_k,q_k$ positive integers. Following \cite {CP}, we introduce the branched cover

\[
\begin{array}{cccc}
\pi:& \mathbb{D}^d \times \D ^{n-d} & \longrightarrow & \mathbb{D}^d \times \D ^{n-d} \\
&(z_1, \ldots, z_d, z_{d+1}, \ldots, z_n) & \longmapsto & (z_1^{q_1}, \ldots, z_d^{q_d}, z_{d+1}, \ldots, z_n)
\end{array}
\]
It branches precisely along $\Delta$, the support of $D$; if we denote by $w$ the coordinates upstairs, then the natural meromorphic vector fields to work with are \[X_k:= \frac{1}{q_k} w_k^{1-p_k} \frac{\d}{\d w_k} \quad (k=1 \ldots d) \] and the usual ones $X_k:= \frac{\d}{\d w_{k}}$ for $k>d$. Indeed, for $k=1,\dots, d$ the vector
fields $X_k$ defined above are precisely the multi-valued vectors  
\[w_k^{1-\beta_k} \frac{\d}{\d w_k} \quad (k=1 \ldots d) \]
expressed in the singular coordinates induced by the ramified cover.

Moreover, we have
\[\omb:=\pi^* \omc = i \sum_{k=1}^d q_k^2 |w_k|^{2(p_k-1)} dw_k \wedge d \bar w_k +i\sum_{k=d+1}^n  dw_k \wedge d \bar w_k \]
which is a cone metric with angles $2\pi p_k$ along $(w_k=0)$.
Hence, unless the coefficients of $D$ are of orbifold type (i.e. $p_k=1$ for all $k$), this metric will be degenerate near $\Delta$. We write $\tilde g_{\beta}$ and $\tilde d_{\beta}$ respectively for the Riemannian metric attached with $\omb$ on $\D^n\setminus \Delta$ and its induced distance. This distance $\tilde d_{\beta}$ gives rise to spaces $\mathscr C^{\alpha, \tilde \beta}=\{f; \, \sup_{\D^n} |f|+\sup_{x\neq y}\frac{|f(x)-f(y)|}{\tilde d_{\beta}(x,y)^{\alpha}}<+\infty \} $ in a similar way as in section \ref{sec:func}.\\

\noindent
Remember that in the orbifold case (i.e. $\beta_k=1/m_k$), a function $f$ is said to have orbifold $\mathscr C^{k,\alpha}$ regularity if its pull-back $\pi^*f$ by the ramified cover $\pi$ is in $\mathscr C^{k,\alpha}$ in the usual sense. 
So we extend this definition to the case of rational coefficients in the following way:

\begin{defi}
\label{lem:func}
Let $f$ (resp. $\tau$, $\sigma$) be a bounded function (resp. $(1,0)$-form, $(1,1)$-form) on $\D^n\setminus \Delta$. Then we say that
\begin{itemize}
\item[$\cdotp$] $f \in \cab  $ if $  \pi^* f \in \mathscr C^{\alpha, \tilde \beta}$ ,
\item[$\cdotp$] $ \tau \in \cab  $ if for all $k$, we have $ \pi^*\tau(X_k) \in \mathscr C^{\alpha, \tilde \beta}$ ,
\item[$\cdotp$] $ \sigma \in \cab  $ if for all $k$ and $l$, we have $\pi^*\sigma(X_k, X_l) \in \mathscr C^{\alpha, \tilde \beta}$.
\end{itemize}
Finally, we let
 \[\mathscr C^{2,\alpha, \beta}:=\{f\in L^{\infty}(\D^n);  f,\d f, \d\bar \d f \in \cab \}\]
and define the associated norm in the natural way.
\end{defi}

We used the same symbol to designate the space of Hölder continuous functions both in the sense of Donaldson and in the sense of the covers. This is legitimated by the following observation: 
\begin{lemm}
\label{lem:func}
The space $\cab$ and $\mathscr C^{\alpha, \beta}_{\rm Don}$ coincide. More precisely, if $f$ is a bounded function on $\D^n\setminus \Delta$, then we have the equality $||f ||_{\cab}=|| f ||_{ \mathscr C^{\alpha, \beta}_{\rm Don}}$ 
\end{lemm}

\begin{proof}
By definition, $\pi$ is a local isometry from $(\D^n\setminus \Delta, \omb)$ to $(\D^n\setminus \Delta, \omc)$. We claim that 
\[\sup_{x\neq y}\frac{|f(x)-f(y)|}{d_{\beta}(x,y)^{\alpha}} = \sup_{x\neq y}\frac{|f(\pi(x))-f(\pi(y))|}{\tilde d_{\beta}(x,y)^{\alpha}}\]
which would prove the first item. Indeed, choose $x, y$ two disctinct points. As $\pi$ is distance non-increasing, $d_{\beta}(\pi(x),\pi(y)) \le  \tilde d_{\beta}(x,y)$ so that  $ \frac{|f(\pi(x))-f(\pi(y))|}{d_{\beta}(\pi(x),\pi(y))^{\alpha}} \ge \frac{|f(\pi(x))-f(\pi(y))|}{\tilde d_{\beta}(x,y)^{\alpha}}$ and passing to the sup, we get the first inequality in the claim.

\noindent
Now, given $x,y$, there exist $x',y'$ such that $x=\pi(x')$, $y=\pi(y')$ and $d_{\beta}(x,y)=  \tilde d_{\beta}(x',y')$ as $\pi$ is an surjective \textit{isometry} from some open subset of $\D^n\setminus \Delta$ to $\D^n\setminus \Delta$ endowed with the suitable metrics. Then  $\frac{|f(x)-f(y)|}{d_{\beta}(x,y)^{\alpha}} = \frac{|f(\pi(x'))-f(\pi(y'))|}{\tilde d_{\beta}(x',y')^{\alpha}}$, which gives the second inequality when passing to the supremum on $x,y$. 

\end{proof}

However, as we hinted at the end of \S \ref{sec:don}, the higher regularity function spaces won't coincide, essentially because the function $z/|z|$ is not Hölder continuous. So if we have a function $f\in \cdab$ then from the arguments above the "pure derivatives" $|z_k|^{2(1-\beta_k)} \frac{\d^2f}{\d z_k \d \bar z_k}$ ($k=1\cdots d$) or $\frac{\d^2f}{\d z_k \d \bar z_k}$ (if $k>d$) belong to $\cab$ but we cannot conclude for the mixed derivatives of the type $|z_k|^{1-\beta_k} \frac{\d f}{\d z_k}$, $|z_k|^{1-\beta_k} \frac{\d^2 f}{\d z_k \d \bar z_l}$ (if $l>d$) and $|z_k|^{1-\beta_k} |z_l|^{1-\beta_l}  \frac{\d^2f}{\d z_k \d \bar z_l}$ (if $l\le d)$. 
One can still say the following, that express how close a function $f\in \cdab$ is to be in $\cdab_{\rm Don}$:

\begin{lemm}
\label{lem:func2}
Let $f\in \cdab$, then:
\begin{enumerate}
\item[$\cdotp$] For each $k= 1, \ldots, d$, $|z_k|^{1-\beta_k} \left| \frac{\d f}{\d z_k} \right| \in \mathscr \cab$
\item[$\cdotp$] For each $k= 1, \ldots, d$ and $l= d+1, \ldots, n$, $|z_k|^{1-\beta_k} \left|\frac{\d^2 f}{\d z_k \d \bar z_l} \right| \in \cab$
\item[$\cdotp$] For each $k,l= 1, \ldots, d$, $|z_k|^{1-\beta_k} |z_l|^{1-\beta_l}  \left|\frac{\d^2f}{\d z_k \d \bar z_l} \right| \in \mathscr \cab$
\end{enumerate}
\end{lemm}

\begin{proof}
Thanks to \ref{lem:func}, it suffices to show the conic Hölder continuity of these functions when pulled-back by $\pi$. 
But $\pi^*\left(|z_k|^{1-\beta_k} \left| \frac{\d f}{\d z_k} \right|\right) = |z_k|^{q-p}\left| \frac{\d f}{\d z_k}(\pi(z)) \right|  = |\pi^*\d f(X_k)|$ and $\pi^*\d f(X_k)$ is in $\mathscr C^{\alpha, \tilde\beta}$ as $f\in \cdab$. But by the reverse triangle inequality, the modulus of an Hölder continuous function is Hölder continuous too. The other items are very similar. 
\end{proof}

\subsubsection{Real coefficients}
Now we want to extend the definition of the (higher regularity) Hölder spaces defined above to arbitrary real numbers $\beta_k \in (0,1)$. Taking covers is no more possible, so if we want to rest on such a definition, we will need a limiting process. 

Observing that $\mathscr C^{\alpha, \beta} \subset \mathscr C^{\alpha, \beta'}$ as long as $\beta>\beta'$, one could say for instance that a function $f$ is in $\cdab$ if for all rational numbers $\beta'<\beta$, $f$ is in $\mathscr C^{2,\alpha, \beta'}$ with $||f||_{\mathscr C^{2,\alpha, \beta'}}\le C$ for a uniform $C$. However, although this is completely consistent for Hölder continuity, this approach is not suited for higher regularity. Indeed, take the most basic example $f(z)=|z|^{2\beta}$, then for $\beta' < \beta$, we have $|z|^{2(1-\beta')}
\frac{\d ^2f}{\d z \d \bar z}= |z|^{2(\beta-\beta')}$ which is not in $\mathscr C^{\alpha, \beta'}$ as soon as $\beta-\beta' < \alpha/2$.
So we would rather say here that given any sequence of rational numbers $\beta_n$ converging to $\beta$, $f$ is the limit (in the $\mathscr C^{\infty}_{\rm loc}$ topology of $\mathbb C^*$) of the functions $f_n(z)=|z|^{2\beta_n}$ belonging to  $\mathscr C^{2,1, \beta_n}$ and satisfying $||f_n||_{\mathscr C^{2,1, \beta_n}}\le C$ for some uniform $C$.

The previous example legimates the following definition:

\begin{defi}
Let $f$ be a bounded function on $\D^n\setminus \Delta$. Then we say that $f\in \cdab$ if for all $k=1,\ldots, d$, there exists a sequence of rational numbers $(\beta_{k,n})_{n\in \N}$ converging to $\beta_k$ such that there is a sequence of functions $f_n \in \mathscr C^{2,\alpha, \beta_{\cdotp,n}}$ converging to $f$ in $\mathscr C^{\infty}_{\rm loc}(\DD)$ and satisfying:
\[\sup_{n\in \N} ||f_n||_{\mathscr C^{2,1, \beta_{\cdotp, n}}}\ < \infty \]
\end{defi}

\noindent
This definition is unfortunately a little bit complicated, and in particular it is not clear from here how to endow $\cdab$ with a reasonable norm. However, it will particularly adapted to our regularity problem for Monge-Ampère equations.

\subsection{Geometry of the cone metric}
\label{sec:geom}

In this section, we collect a few facts about the geometry of $(\CC^n \setminus \Delta, g_{\beta})$, and it will be important that we allow the angles $\beta_k$ to take any value in the range $(0,+\infty)$. We introduce the following map:

$$\begin{array}{ccccc}
\Psi & : & (\DD, \om_{\rm eucl}) & \longrightarrow & (\DD, \omb) \\
 & & (z_1, \ldots, z_n) & \longmapsto & (c_1|z_1|^{\frac{1}{\beta_1}-1}z_1, \ldots, c_d|z_d|^{\frac{1}{\beta_d}-1}z_d, z_{d+1}, \ldots, z_n) \\
\end{array}$$
where $c_k=\beta_k^{1/2\beta_k}$. Then a basic computation shows that $\Psi$ induces is an isometric diffeomorphism between $(\DD, \om_{\rm eucl})$ and $ (\DD, \omb)$. From that we easily deduce:

\begin{lemm}
\label{lem:inj}
Let $p \in \CC^n \setminus \Delta$. Then the injectivity radius $\mathrm{inj}_{g_{\beta}}(p)$ of $g_{\beta}$ at $p$ equals $d_{\beta}(p,\Delta)$.
\end{lemm}

\begin{proof}
It is enough to to prove it for $(\DD, \om_{\rm eucl})$. Let $r=d(p,\Delta)$; the exponential map at $p$ is well-defined on $B(p,r)\subset \R^{2n}$ and is clearly a diffeomorphism. As this map cannot be extended to a larger region, we get the desired result. 
\end{proof}

\noindent

In the course of the proof of Theorem B, we will have to deal with two types of balls: 
first, the geodesic balls $B(p,r)$ where $p\in \DD$ and $r<d_{\beta}(p, \Delta)$, i.e.
\[B(p,r)= \{z\in \DD; d_{\beta}(z,p)<r\}\]
The other to be considered are "balls" centered a point $p\in \Delta$, which consists of points in $\DD$ with (conic) distance less than $r$ to a $p$ \--- without restriction on $r$. In the latter case, one can assume without loss of generality that $p=0$, so that \[B(0,r)= \{z\in \DD; \, (r_1^{\beta_1}/\beta_1)^2+\cdots +  (r_d^{\beta_d}/\beta_d)^2 +\sum_{k>d} r_k^2 < r^2 \}\]

\begin{lemm}
\label{lem:vol}
Let $p\in \D^n$, and set $V(r)=\mathrm{vol}_{\omc}(B(p,r))$. Then:
\[V(r) =
\begin{cases}
\frac{\pi^n}{n!} \, r^{2n}&\mbox{if } p\notin \Delta\\
\frac{\beta_1 \cdots \beta_d \, \pi^n }{n!}\, r^{2n} &\mbox{else}
\end{cases}
\]
\end{lemm}

\begin{proof}
In the first case, if $q=\Psi^{-1}(p)$, then $\Psi$ induces a isometry between $B_{\rm eucl}(q,r)$ and $B(r)$, so we are done. The remaining case follows from a straightforward computation.
\end{proof}

\noindent
We have the following Sobolev inequality for $g_{\beta}$:

\begin{prop}
%\phantomsection
\label{sobolev2}
Let $B(r)=B(p,r)$ be a ball as above, and let $f$ be a bounded smooth function on $\DD$ with compact support in $\bar B(r)$. Then: 
\[\left(\int_{B(r)} |f|^{\frac{2n}{n-1}}\omc^n \right)^{\frac{n-1}{n}}\le C V(r)^{-1/n}r^2 \int_{B(r)} |\nabla f|^2 \omc^n \]
for some constant $C$ depending only on $n$.
\end{prop}

%Here again, the constant $C$ may be chosen independently of the angles as long as they stay away from $0$. 
\noindent
Of course, the gradient is computed here with respect to $g_{\beta}$, so that the term inside the integral of the right had side equals (up to a constant) $df\wedge d^c f \wedge \omc ^{n-1}$.

\begin{proof}
We first consider the case $p\notin \Delta$ so that $B(r)$ is a geodesic ball for $(\DD, \omc)$.
Then if $q=\Psi^{-1}(p)$, $\Psi$ induces a isometry between the euclidian ball $B_{\rm eucl}(q,r)$ and $B(r)$. Therefore Sobolev inequality for euclidian balls gives immediately Sobolev inequality for conic balls (as $V(r)^{-1/n}r^2$ is a universal constant in that case). \\

We treat next the complementary case $p\in \Delta$, so that by definition 
$$B(r)=B(0,r)=\{z\in \DD; \, (r_1^{\beta_1}/\beta_1)^2+\cdots +  (r_d^{\beta_d}/\beta_d)^2 +\sum_{k>d} r_k^2 < r^2 \}.$$ 
Then $B'(r):=\Psi^{-1}(B(r))= \{z\in \DD; \, 
(c_1^{\beta_1} r_1/\beta_1)^2+\cdots +  (c_d^{\beta_d}r_d /\beta_d)^2 
 +\sum_{k>d} r_k^2 < r^2 \}$. 
%In particular, \[\vol_{\rm eucl}(B'(r))= \left(\prod_k \beta_k c_k^{-\beta_k} \right) \frac{\pi^n}{n!} \, r^{2n}.\]

But we know that Sobolev inequality (for the euclidian metric) holds for domains like $\DD$ or $B(r)$. Indeed, $\Delta \subset \D^n$ admits a family of cut-off functions $\chi_{\ep}$ such that $||\nabla \chi_{\ep}||_{L^2}$ goes to $0$ when $\ep \to 0$; for instance one may take $\chi_{\ep}=\xi_{\ep}(\log \log \frac{1}{r_1 \cdots r_d})$, where $\xi_{\ep}$ is zero on $[0,\frac{1}{\ep}]$, equals one on $[\frac{1}{\ep}+1, +\infty[$, and is affine $[\frac{1}{\ep},\frac{1}{\ep}+1]$. \\

Therefore, if $F= \Psi^*f$, and $\om_{\rm eucl}$ denotes the euclidian metric, then performing a linear change of variable twice yields:

\begin{eqnarray*}
\left(\int_{B(r)} |f|^{\frac{2n}{n-1}}\omc^n \right)^{\frac{n-1}{n}} & = & \left(\int_{B'(r)} |F|^{\frac{2n}{n-1}}\om_{\rm eucl}^n \right)^{\frac{n-1}{n}} \\
& = & \left( \prod_k \beta_k^2 c_k^{-2\beta_k} \right)^{\frac{n-1}{n}}\left(\int_{B_{\rm eucl}(r)} |F|^{\frac{2n}{n-1}}\om_{\rm eucl}^n \right)^{\frac{n-1}{n}} \\
& \le & C \left( \prod_k \beta_k^2 c_k^{-2\beta_k} \right)^{\frac{n-1}{n}} \int_{B_{\rm eucl}(r)} |\nabla F|^2 \om_{\rm eucl}^n \\
& = & C\left( \prod_k \beta_k^2 c_k^{-2\beta_k} \right)^{-1/n}  \int_{B'(r)} |\nabla F|^2 \om_{\rm eucl}^n \\
 & = & C\left( \prod_k \beta_k^2 c_k^{-2\beta_k} \right)^{-1/n} \int_{B(r)} |\nabla f|^2 \omc^n \\
 & = & C' V(r)^{-1/n} r^2\int_{B(r)} |\nabla f|^2 \omc^n
\end{eqnarray*}
as $c_{k}^{-2\beta_k}= \beta_k^{-1}$ and $V(r)= C(n) \, \beta_1 \cdots \beta_k r^{2n}$.
\end{proof}

\begin{rema}
\label{rem:sob}
This inequality will still hold \textit{with the same constant} $C$ if we replace the metric $\omc$ by $i\sum_{k=1}^d \mu_k |z_k|^{2(\beta_k-1)} dz_k \wedge d \bar z_k +i\sum_{k=d+1}^n dz_k \wedge d \bar z_k $ for some positive constants $\mu_k$. Indeed, if we change $c_k$ into $\mu_k^{-1/2\beta_k} c_k$ in the definition of $\Psi$, we still get an isometry between the rescaled metric and the euclidian one, and the same volume estimate as in Lemma \ref{lem:vol} remains true; therefore the previous argument can be run again. This observation is important as we will eventually apply this particular Sobolev inequality to the pull-back of $\omc$ by the covering map, cf next section. We should also notice that this inequality gets sharper than the standard Sobolev inequality as the angles tend to $+\infty$, which will precisely be our case later in the proof.
\end{rema}

\subsection{Conic Harnack inequality}
In this subsection we establish the main technical tool of the proof, namely the Harnack inequality
in conic setting. A first step is to show that one can perform integration by parts 
with respect to the cone metric.

\begin{comment}
Our references for this section are \cite[pp.100-113]{Siu}, \cite[Theorem 8.18]{Gilb}, \cite{SC}, \cite[Theorem 4.15]{HL97}.
We will work on $\mathbb D^ n$ that we endow with the degenerate metric $\omb:=\pi^ * \omc$. If we pull back $\vp$ by $\pi$, we end up with a bounded function $u:=\vp \circ \pi$ on $\D^ n$ which is $\mathscr{C}^2$ outside of $\Delta$, and satisfies
\begin{equation}
\label{eq:mad}
(\ddc u)^n = \prod_{k=1}^d q_k^2|w_k|^{2(p_k-1)}e^{\mu u+ F}dV
\end{equation}
where $F=f\circ \pi$.\\
\end{comment}

\subsubsection{Integration by parts}
\label{sec:ipp}

The context is the following: let $u,v$ be bounded smooth functions on $\DD$ that can be written as differences of functions whose conic laplacian is uniformly bounded from below, i.e. $u=u_1-u_2$ with $\Delta_{\omc} u_i \ge -C$ on $\DD$, and similarly with $v$. Let us set $T= \omc^{n-1}$, which is a $(n-1, n-1)$ closed positive current on $\D$. 

\noindent
Then $u \, \ddc v \wedge T$ and $du \wedge d^c v \wedge T$ are two smooth currents on $\DD$ of degree $2n$, and can be viewed as complex measures on that set. 

\begin{prop}
\label{ipp}
These two currents have finite mass near $\Delta$, and if $\eta$ is a cut-off function with compact support in $\D^n$, then:
\[\int_{\DD} \eta u \, \ddc v \wedge T = -\int_{\DD} d(\eta u) \wedge d^c v \wedge T \]
\end{prop}

\begin{proof}
The main difficulty here is that we do not deal with quasi psh function but only quasi subharmonic functions with respect to a singular metric. Because of this, we cannot use directly e.g. \cite[Theorem 1.14]{BEGZ}. However, in order to establish this kind of results, the key point is to have a regularization procedure. We treat this in detail along the following subsections.\\

\noindent
\textbf{Step 1. The cut-off function.}

\noindent
Fortunately, the cone geometry is rather well understood, and we have at our disposal nice cut-off functions as shown in \cite[\S 9]{CGP}. Let us recall briefly their construction. Let $\rho = \log (- \log \prod |z_j|^2)$, and, for all $\delta>0$, $\xi_{\delta}: \R_+\to \R_+$ be a smooth function equal to $0$ on $[0, 1/\delta]$ and to $1$ on $[1+1/\delta, +\infty[$ (and always between $0$ and $1$). We set $\chi_{\delta}=1- \xi_{\delta}(\rho)$, it is $1$ whenever $\prod |z_j| \ge e^{-e^{1/\delta}}$ and $0$ if $\prod |z_j| \le e^{-e^{1+1/\delta}}$. Then one can check as in \cite{CGP} that: 
\[\left|\! \left| \log \left(\prod |z_j| \right)\, \ddc \chi_{\delta} \wedge T \right | \! \right |_{\D^n} < +\infty\]
and
\[\int_{\D^n}  d \chi_{\delta} \wedge d^c \chi_{\delta} \wedge T \, \,  \underset{\delta \to \, \,  0}{\longrightarrow} 0 \]
In other words, the gradient of the cut-off function tends to $0$ with respect to the $L^2$ topology induced by $\omc^n$. \\

\noindent
\textbf{Step 2. The regularization.}

\noindent
In this section, we assume that the function $u$ is $\omc$-subharmonic, i.e. $\ddc u \wedge T \ge 0$ (we always work outside of $\Delta$). We want to show that this positive current has finite mass on $X$. To do that, we need to regularize $u$. First, we define $v_{\ep}= u+ \ep \log \prod |z_j|$. In that way we made $u$ extend continously (with value $-\infty$) to the whole of $\D$. To make it smooth and still preserving its subharmonicity, we set $u_{\ep}= \max_{\ep}(v_{\ep}, -M)$ where $M> \sup |u|+1$, and $\max_{\ep}$ means a regularized maximum. In that way, $u_{\ep}$ equals $-M$ near $\Delta$, and a bit further, we convolute it with a smoothing kernel. This operation will preserve the subharmonicity as $\omc$ is smooth on $\DD$. Indeed, the maximum of two $\omc$-subharmonic functions is still $\omc$-subharmonic (write the maximum function as a supremum of affine functions and use the characterization of weak subharmonic functions) and it is continuous; therefore we may use for instance the results of \cite[Corollary 1, p. 66]{GW2} to regularize it.
In the end, we obtain a smooth $\omc$-subharmonic function $u_{\ep}$ that converges smoothly to $u$ on each compact subset of $\DD$ (this is still a consequence of \cite{GW2}). Moreover
\begin{eqnarray*}
\int_{\DD} \eta \, \ddc u_{\ep} \wedge T & = & \lim_{\delta \to 0} \int_{\D^n}\eta \chi_{\delta} \, \ddc u_{\ep} \wedge T\\ 
&= &  \lim_{\delta \to 0} \int_{\D^n} \uep \, \ddc (\eta\chid) \wedge T 
\end{eqnarray*}
As $|\uep | \le C - \ep \log \prod |z_j|$, then using Cauchy-Schwarz inequality, we can dominate the previous quantity up to a constant by
\[|| \ddc \chid \wedge T|| +||T||+||d\chid \wedge d^c\chid \wedge T||^{1/2}+ \ep ||F \ddc \chid \wedge T||+\] 
\[\ep ||F\, T|| + \ep ||d\chid \wedge d^c\chid \wedge T||^{1/2} || F^2 \, T||^{1/2}\]
where $F=\log \prod |z_j|$ and the norms are taken over $\D^n$. Using the properties of $\chid$ recalled above and the fact that $FT$ as well as $F^2T$ have finite mass (this is a straightforward computation), we conclude that the mass of $\eta \ddc \ue \wedge T$ on $\DD$ is uniformly bounded when $\ep$ goes to $0$. As this currents converge (smoothly) to $\eta \ddc u \wedge T$ on $\DD$, we infer that this last current has finite mass on $\DD$: 
\[\int_{\DD} \eta \,\ddc u \wedge T < +\infty\]
As a consequence, $u$ has a gradient in $L^2$: 
\[\int_{\DD} \eta \, du \wedge d^c u \wedge T < +\infty.\]

Indeed, assume $u$ non-negative. Then $u^2$ is suhbarmonic and
$\ddc u^2 = 2u \ddc u +2 du \wedge d^c u$. As $u$ is bounded, $u \ddc u \wedge T$ has finite mass, and so does $\ddc u^2 \wedge T$ as we have showed above; this proves the claim.

We note that these results hold more generally if the function $u$ is a difference of quasi-subharmonic functions. Indeed, if $u$ satisfies merely $\Delta_{\omc}u \ge -C$, then if $u_0:=A |z_1|^{2\beta_1}$, the function $u+u_0$ is subharmonic for $A$ big enough, hence $\ddc u \wedge T =\ddc (u+u_0) \wedge T - \ddc u_0 \wedge T $ is a difference of measures with finite total mass, so it has finite mass too. A similar result holds for differences of quasi-subharmonic functions.\\

\noindent
\textbf{Step 3. Integration by parts.}

\noindent
We consider two (non-negative) $\omc$-subharmonic functions $u,v$ as in the statement \ref{ipp}. 
We know that $u \ddc v \wedge T$ and $du \wedge d^c v \wedge T$ have finite norms on $X$. More precisely: 
\begin{eqnarray*}
 \int_{\D^n}\eta \chi_{\delta}u \ddc v \wedge T &= & -\int_{\D^n} d(u \eta\chid) \wedge d^c v \wedge T \\
 & = & -\int_{\D^n}  \chid \, d(\eta u) \wedge d^c v \wedge T - \int_{\D^n}  \eta u \, d\chid \wedge d^c v \wedge T
\end{eqnarray*}
But the last integral is controlled by $||u||_{\infty} ||\nabla v||^{1/2}||\nabla \chid ||^{1/2}$ which tends to $0$ when $\delta \to 0$. This concludes the proof of Proposition \ref{ipp}.

\end{proof}

\begin{rema}
We should notice that we do not really use the precise expression of the cone metric in the previous arguments. The proof equally works for instance for the pull-back $\omb$ of the cone metric $\omc$: actually we only needed the metric to have bounded potentials and that $\Delta \subset \D^n$ admits a cut-off function whose gradient with respect to the metric is small in $L^2$ norm.
\end{rema}

\subsubsection{Harnack inequality}

%We will divide the proof into two parts: first we prove the Harnack inequality in the conic setting, and then we adapt the convexity arguments of Evans Krylov theorem using the branched cover.\\

As we will see at the end of the proof of Theorem B, it is not enough to establish the conic version of Harnack inequality for geodesic balls. It is indispensable to show it for balls centered at a point of the divisor; this is the content of the current subsection. 
\smallskip

\noindent
We recall that we have introduced in \ref{sec:geom} two kinds of balls $B(r)$ that are either geodesic balls for $(\DD, \omc)$ or "balls" centered at a point of $\Delta$. We have denoted by $V(r)=\mathrm{vol}_{\omc}(B(r))$ the volume of these balls, which is explicitly computed in Lemma \ref{lem:vol}. In this setting, we have the following Harnack inequality.

\begin{theo}
\label{harnack}
Let $B(r)$ be any ball as above such that $B(3r) \subset \DD$. Assume that $v$ is a bounded non-negative smooth function on $\DD$ satisfying the inequality
\[\Delta_{\omc} v \le \theta\]
for some bounded smooth function $\theta$. Let $q>n$; then there exists $p>0$ such that:
\[V(r)^{-1/p} ||v||_{L^p(B(2r))} \le C \left( \inf_{B(r)} v + r^2\, V(r)^{-1/q} ||\theta||_{L^q(B(2r))}\right)\]
where $C$ is a constant depending only on $n,p,q$ and $\omc$.
\end{theo}

The $L^p$ spaces involved here are defined using the volume form induced by $\omc$, i.e. $\omc^n/n!$. This inequality also holds for any metric $\om$ quasi-isometric to $\omc$ (cf Remark \ref{rem:har}); this will be important in what follows. Let us also stress that the key point is that this estimate is uniformly satisfied for all ball $B(3r)\subset \DD$, independently of their radius.

%We chose to state the result in terms of the degenerate metric $\omc$, but the proof also works for the lift of the cone metric $\omb$ 

\begin{proof}
We will essentially follow \cite[\S 5, p 107-113]{Siu} and \cite[Theorem 4.15]{HL97}. There are essentially three important facts in the standard case that need to be modified in order to accommodate them to our degenerate setting.

\begin{itemize}
\item In the usual proofs of this inequality, only the case $r=1$ is treated as the general case can be deduced from this one by a change of variable. Here we could use this idea using a suitable change of variable, but it would only work for the balls centered at the divisor, and even in that case, it would not really simplify the proof, so we decided to choose the uniform framework of radius $r$ balls.
This requires a finer control on the constants involved; the precise form of Sobolev inequality that we obtained in \ref{sobolev2} will be crucial to get uniform estimates valid even when $p,q$ go to infinity. 
\item We also need to explain why we are allowed to perform the integrations by parts for the balls centered at a point of $\Delta$ (this will be a consequence of Proposition \ref{ipp}).

\item Finally we have to use Sobolev inequality in our context and also we have to explain how to avoid the use of John-Nirenberg inequality --since this later ingredient of the classical case proof does not seem to be
obvious to establish in the conic setting.\\
\end{itemize}

\noindent 
To start with, we observe that one can always find a cut-off function $\eta$ with support in $B(R)$
and which equals 1 on $B(r)$ such that $|D \eta  |  \le C (R-r)^{-1}$. This is clear for the geodesic balls working in the normal coordinates given by Gauss lemma, and this is not much more difficult for the balls centered at $0$. 

We will always consider \textit{normalized measures}, i.e. every $L^p$ spaces and norms over $B(r)$ that we will consider will be taken with respect to $d\mu_r:=dV_{\omc}/V(r)$. To get a normalized Sobolev inequality, we multiply each side of inequality \eqref{sobolev2} by $V(r)^{\frac{1}{n}-1}$; the resulting inequality reads as follows
\begin{equation}
\label{sobolev3}
||f||_{L^{\frac{2n}{n-1}}}^2 \le C r^2 ||\nabla f||_{L^2}^2.
\end{equation}
\vspace{1mm}

$\bullet$ Let $A$ be the $L^q$ norm of $\theta$ over $B(2r)$, $w=v+r^2A$, and $\nu>0$. 
A straightforward computation using the elliptic inequality satisfied by $w$ leads to 
\begin{equation}
\label{eq:in}
\Delta_{\omc} w^{-\nu} \ge - \nu \xi w^{-\nu}
\end{equation}
where $\xi=\theta/(v+r^2 A)$, so that $||\xi||_{L^q} \le r^{-2}$.

Let $r<r_1<r_2< 2r$ be two numbers to be determined later. We choose a cut-off function $\eta$ equal to $1$ on $B(r_1r)$ and with support on $B(r_2r)$. Multiplying both sides of \eqref{eq:in} by $\eta^ 2 w^ {-\nu}$ and integrating by parts (as Proposition \ref{ipp} allows us to do, since the $\omc$-laplacian of this function is bounded from below), we get:
\[ \int_{B(2r)} \eta^ 2 |\nabla w^ {-\nu}|^2 d\mu_r+\int_{B(2r)}2\eta w^{-\nu} \nabla \eta \cdotp \nabla  w^{-\nu} d\mu_r \le  \int_{B(2r)} \nu \xi \eta^2 w^{-2\nu} d\mu_r \]
As a consequence, we obtain:
\[||D(\eta w^{-\nu})||_{L^2}^2 \le C \left( ||w^{-\nu} D\eta|||_{L^2}^2 + ||\nu \xi \eta^2 w^{-2\nu}||_{L^1}  \right)\]
Then using Sobolev \eqref{sobolev3} and Hölder's inequalities, we infer
\begin{equation}
\label{eq0}
||\eta w^{-\nu}||^2_{L^{\frac{2n}{n-1}}} \le C r^2 \left( ||w^{-\nu} D\eta||_{L^2}^2 + \nu || \xi||_{L^ q} || \eta w^{-\nu}||^ 2_{L^{\frac{2q}{q-1}}} \right). 
\end{equation}
Now an interpolation inequality yields for every $\ep>0$ the following relation
\[ || \eta w^{-\nu}||^ 2_{L^{\frac{2q}{q-1}}} \le 2 \ep^ 2   || \eta w^{-\nu}||^ 2_{L^{\frac{2n}{n-1}}} +2\ep ^ {-\frac{2n}{q-n}}||\eta w^{-\nu}||^2_{L^{2}} ;\]
we choose $\ep$ so that $C r^2 \nu || \xi||_{L^ q}  2 \ep^ 2 =1/2$, and plug it in the previous inequality. Combining this with \eqref{eq0}, we obtain:
\[||\eta w^{-\nu}||^2_{L^{\frac{2n}{n-1}}} \le 2 Cr^2  ||w^{-\nu} D\eta|||_{L^2}^2 + \left(4 \nu Cr^2 ||\xi||_{L^q}\right)^{\frac{q}{q-n}}||\eta w^{-\nu}||^2_{L^{2}}. \]
Since we have $|D\eta| \le C r^{-1}(r_2-r_1)^{-1}$ and $||\xi||_{L^q} \le r^{-2}$, we obtain:
\begin{equation}
\label{har0}
||w^{-\nu}||_{L^{\frac{2n}{n-1}}(B(r_1r))} \le C \frac{(1+\nu)^{\frac{q}{2(q-n)}}}{r_2-r_1}||w^{-\nu}||_{L^{2}(B(r_2r))}
\end{equation}
where $C$ \textit{does not depend on} $r$. Take any $p>0$. After choosing these numbers 
in an appropriate way and iterating the process (cf. \cite[p. 110]{Siu}), we end up with:
\begin{equation}
\label{har1}
\sup_{B(r)} w^ {-1} \le C ||w^{-1}||_{L^p(B(2r))}
\end{equation}
\begin{comment}
Now, given any $0<p_0<p$, iterating \eqref{har0} yields
\[||w^{-1}||_{L^p(B(2r))} \le C ||w^{-1}||_{L^{p_0}(B(3r))}\]
\end{comment}
If we can show that there exists $p_0>0$ such that 
$$||w^{-1}||_{L^{p_0}(B(3r))} \le C||w||_{L^{p_0}(B(3r))}^ {-1},$$ 
then by combining this inequality with \eqref{har1} we obtain $\inf_{B(r)} w \ge C^{-1}||w||_{L^{p_0}(B(3r))}$ and therefore it would end the proof as $w=v+r^2 ||\theta||_{L^q}$.\\

So we need now to show the existence of $p_0>0$ such that $||w^{-1}||_{L^{p_0}(B(3r))} \le C||w||_{L^{p_0}(B(3r))}^ {-1}$. Usually, the proof of this estimate involves the John-Nirenberg inequality (cf \cite{Siu, Gilb}), which in our context does not seem to be an easy fact to prove. Fortunately, 
it turns out that it is possible to avoid using it: as the argument given in \cite[Theorem 4.15, pp 98-103]{HL97}) shows it, one can obtain the desired estimate by only using Hölder's, Young's, Poincaré's and Sobolev's inequalities.

For some good reasons that we have already invoked, we have to work on the ball $B(r)$ (and not $B(1)$), so we will briefly indicate next the necessary modifications we have to operate with respect to the proof of \cite[Theorem 4.15, pp 98-103]{HL97}.\\

$\bullet$ Set $\psi:=\log w-\ib \log w$ (here again the measure is normalized). It is enough for our purpose to show that for some $p_0>0$, $\ib e^{p_0 |\psi|} \le C$ for some $C$ independent of $r$. Therefore we have to estimate $\ib |\psi|^{\gamma}$ for all positive integer $\gamma$. 

The first step of \cite{HL97} can be adapted with (almost) no modification to guarantee that $\ib \psi^2 \le C$. Then we have to choose as a test function $\eta^2 |\psi|^{2 \gamma}$, where $\eta$ is a cut-off function as above and $\gamma \ge 2$ is an integer. We can use the exact same arguments of \cite[pp100-101]{HL97} to end up with:
\[ ||D(\eta |\psi|^{\gamma})||_{L^2}^2 \le C \left [ (2\gamma)^{2\gamma} ||\,\eta |D\psi|\,||_{L^2}^2+\gamma^2 ||\, |D\eta| |\psi|^{\gamma} ||_{L^2}^2+\gamma ||\xi||_{L^q}||\eta |\psi|^{\gamma}||_{L^{\frac{2q}{q-1}}}^2\right] \]
the norms being taken over $B(r)$. By interpolation and Sobolev inequalities, we obtain for every $\ep>0$:
\[ ||\eta |\psi|^{\gamma}||_{L^{\frac{2q}{q-1}}}^2 \le Cr^2 \ep^ 2||D(\eta |\psi|^{\gamma})||_{L^2}^2 +C'\ep^ {\frac{-2n}{q-n}} ||\eta |\psi|^{\gamma}||_{L^2}^2  \]
Choosing a suitable $\ep$ (the same as before actually), we get:
\[ \int_{B(r)} |D(\eta |\psi|^{\gamma})|^2 \le C  \left[ (2\gamma)^{2\gamma} \int_{B(r)} |D\eta|^2+\gamma^{\alpha}\int_{B(r)} (|D\eta|^2 + ||\xi||_{L^q} \eta^2) |\psi|^{2\gamma} \right] \]
for some constant $\alpha>0$ depending only on $n,q$. If we apply now Sobolev inequality to $\eta |\psi|^{\gamma}$ and choose $\eta$ to be a cut-off function for $B( r_1 r) \subset B( r_2 r)$ with the same $r_i$'s as before, we end up with:
\[ \left(\int_{B( r_1r)} |\psi|^{\frac{2\gamma n}{n-1}}\right)^{\frac{n-1}{n}} \le \frac{C \gamma^{\alpha}}{(r_2-r_1)^2}\left((2\gamma)^{2\gamma} + \int_{B( r_2r)} |\psi|^{2\gamma} \right) \]
as the constant in the renormalized Sobolev inequality \eqref{sobolev3} behaves like $r^2$ and $||\xi||_{L^q} \le r^{-2}$, an iteration of the process yields
\[ \int_{B( r)} |\psi|^{\gamma} \le (Ce)^{\gamma} \gamma! \]
for all integers $\gamma \ge 1$, hence choosing $p_0=(2Ce)^{-1}$ gives
$\int_{B( r)} e^{p_0 |\psi|} \le 2$ which concludes the proof. 
\end{proof}

\begin{rema}
\label{rem:har}
The previous Harnack inequality holds \textit{with a uniform constant} for all angles $\beta_k \in (0, +\infty)$ as well for the pull-back $\omb$ of $\omc$, as the only ingredients that we used in the proof were: Sobolev inequality (in its refined form given by Proposition \ref{sobolev2}), integration by parts outside of $\Delta$ and also the existence of appropriate cut-off functions. All these properties have been previously shown to be uniformly satisfied by $\omc$ or $\omb$ and more generally by any metric which is quasi-isometric to $\omb$ (cf Remark \ref{rem:sob}), so Harnack inequality will hold true uniformly for all these metrics. 
\end{rema}

\subsection{Evans-Krylov's argument.}
In this part, we will assume that the angles $\beta_k$ are rational numbers. This is needed in order to use the branched covers introduced above.

The usual Evans-Krylov method provides an inequality of the form $\om(R) \le C R^{\alpha}$ where $\om(R)$ is the oscillation on a (geodesic) ball of radius $R$ of the function which we want to prove to be Hölder continuous. This is sufficient to prove the Hölder estimate provided that any two points within distance $R$ are contained in a geodesic ball of radius proportional to $R$. However, this is not the case for the cone geometry (think of two points $x,y \in \DD$ such that $d(x,y) \gg d(x,\Delta)$ and consider Lemma \ref{lem:inj}). It is precisely for this reason that 
from the very beginning of our proof we have considered not only geodesic balls but also balls centered at a point of $\Delta$. \\

Let now $R>0$ be a positive number; from now on we work in a ball $B(R)$ as in the previous section (either a geodesic ball away or a ball centered at $0$).
The usual Evans-Krylov argument consists in combining the concavity of $\log \det$ with Harnack inequalities obtained from the linearization of the (MA) equation. However, it is crucial that the rhs of the MA equation has $\mathscr{C}^2$ regularity and that the solution is "uniformly strictly psh". In our case, these conditions are not fulfilled. The trick consists to work with the twisted differential operators corresponding to the differentiation with respect to the twisted vector fields $X_k$ introduced in section \ref{sec:tw}:
\begin{equation*}
 \dbj = 
\begin{cases}
\frac{1}{q_k} w_k^{1-p_k} \frac{\d}{\d w_k} & \text{if $k=1 \ldots d$} \\
\frac{\d}{\d w_{k}} & \text{if not}
\end{cases} 
\end{equation*} 
and similarly for $\dbkb$. We also define $\db$ and $\dbb$ by $\db f = \sum_k (\dbj f) dw_k$ and $\dbb f =\sum_k (\dbkb f)  d\bar w_k$. \\

$\bullet$
Recall that the pull-back $u=\pi^* \vp$ of the solution $\vp$ satisfies 
\begin{equation*}
%\label{eq:mad1}
(i\d \bar\d u)^n = \prod_{k=1}^d q_k^2|w_k|^{2(p_k-1)}e^{\mu u+ F}dV
\end{equation*}
where $F=f\circ \pi$.
Using the multilinearity of the determinant, it is an easy exercise to check that this Monge-Ampère equation equivalent to:
\begin{equation}
\label{eq:mad2}
(i\db \dbb u)^n = e^{\mu u+F}dV
\end{equation}
and the last equation looks like a non-singular MA equation. Besides, we already know from the assumptions of the theorem that $u$ is uniformly bounded on $\D^ n$ and that $i\db \dbb u$ is quasi-isometric to the euclidian metric on $\DD$, or equivalently that $ \om_u:=i\d \bar\d u$ is quasi-isometric to $\omb$. If we set $\Phi := \log \det$, and $h= \mu u +F$ (as $\det(dV)=1$) then we have:
\begin{equation}
\Phi(i\db \dbb u)=h \label{eq2} 
\end{equation}
\vspace{0.3mm}

\noindent
What do we know about $h$? Firstly, $F=\pi^*f$ is the pull-back of a function of class $\mathscr C^3$ so its twisted derivatives are of class $\mathscr C^{\alpha}$ for all $\alpha \in (0,1)$. As for $\mu u$, we know that it is bounded, but it actually follows from \cite{Kolo2,DDGHKZ} that $\vp$ is in $\mathscr C^{\alpha}\subset \cab$ for some $\alpha>0$, hence $u\in \mathscr C^{\alpha, \tilde \beta}$. Finally, we know  that $i\d \bar \d u $ is quasi-isometric to $\omb$, and therefore $\db \dbb u$ has uniform $\mathscr C^0$ bounds. \\

$\bullet$ 
We consider next a constant twisted vector field $\gamma$ on $\CC^n$, and we differentiate \eqref{eq2} with respect to $\gamma$ and then again with $\bar \gamma$. 
As the twisted differential operators commute, the concavity of $\Phi$ this leads to
\[{g}^{j \bar k} \db_j  \d^{\beta}_{\bar k} w \ge h_{\gamma \bar\gamma}\]
where $w=u_{\gamma \bar\gamma}$ and $(g_{j \bar k})$ are defined by $i\db\dbb u=\sum_{j,k}g_{j\bar k}\, idz_j \wedge d\bar z_k$. We emphasize that we differentiate with the twisted $\db, \dbb$ operators. More precisely, if $\gamma = \sum \gamma_k X_k$, then $h_{\gamma \bar\gamma}:=\sum_{k,l} (\dbj  \d^{\beta}_{\bar l} h) \gamma_k \bar \gamma_l$. \\
%g^{j \bar k} \dbj \dbkb w \ge h_{\gamma \bar\gamma}  \label{eq:ei2}
Basic algebraic manipulations (namely that the inverse matrix of $(z_i \bar z_j a_{ij})_{ij}$ is given by $(\bar z_i^{-1} z_j^{-1} a^{ij})_{ij}$) show that the operators $\sum_{k,l}g^{k \bar l} \db_k  \d^{\beta}_{\bar l}$ and  $\Delta_{\om_u} $ coincide, so that we finally end up with:
\begin{equation}
\Delta_{\om_u} w \ge h_{\gamma \bar\gamma} \label{eq:ei3}
\end{equation}
\vspace{1mm}

$\bullet$
The function $\Phi$ is concave on the set of hermitian definite positive matrices. So, if $y\in \DD$, one may apply a convexity inequality at the point $g(y)=i\db \dbb u(y)$, which thanks to \eqref{eq2} yields for all $x\in \DD$:
\[\Phi_{j\bar k}(g(y))(g_{j \bar k}(y)-g_{j \bar k}(x)) \le h(y)-h(x)\]
where $\Phi_{j\bar k}(g(y)) = (\d \Phi / dx_{j\bar k})(g(y))=g^{j \bar k}(y)$. 
As $g(y)$ is quasi-isometric to the euclidian metric (say with eigenvalues in some fixed interval  $[\lambda, \Lambda] \subset \R_+^*$ independent of $y$), one can apply a basic lemma from linear algebra (cf e.g. \cite[(4.3)]{Siu}) to find vectors $(\gamma_{\nu})_{1\le \nu \le m}$ in $\CC^n$ (depending on $y$) and real numbers $(\tau_{\nu})_{\nu}$ with $\tau_{\nu}\in [\lambda, \Lambda]$ such that $g(y)=\sum \tau_{\nu} \gamma_{\nu} \otimes \bar \gamma_{\nu}$. Setting $w_{\nu}:=u_{\gamma_{\nu} \bar \gamma_{\nu}}$, we get from the previous inequality the following one:
\begin{equation}
\label{in:conv}
\sum_{\nu=1}^m \tau_{\nu}(w_{\nu}(y)-w_{\nu}(x))\le h(y)-h(x)
\end{equation}

$\bullet$
We will combine inequality \eqref{in:conv} with Harnack inequality, i.e. Theorem \ref{harnack}, applied to the elliptic inequality \eqref{eq:ei3}, what Remark \ref{rem:har} allows us to do. Let us set, for $s=1,2$, $m_{s\nu}:=\inf_{B(sR)} w_{\nu}$ and $M_{s\nu}:=\max_{B(sR)} w_{\nu}$. Then for each $\nu$, $M_{2\nu}-w_{\nu}$ is non negative and satisfies an appropriate elliptic inequality by \eqref{eq:ei3}. By cleverly combining \eqref{in:conv} with the estimates that Theorem \ref{harnack} will give us, one can carry on the classic arguments to get for any $\nu, R$ and $\alpha \in (0,1)$:
\begin{eqnarray*}
 \left ( \frac{1}{V(R)} \int_{B(R)}(w_{\nu}(y)-m_{2\nu})^ pdy\right)^ {1/p}& \le &
 C \bigg(\om(2R)-\om(R)+\\
 & & + R^ {\alpha} ||h||_{\mathscr{C}^ {\alpha, \tilde \beta}(B(2R))}+R^2 \sup_{B(2R)} ||\db \dbb h||\bigg) 
\end{eqnarray*}
where $\om(sR)=\sum_{\nu=1}^m \mathrm{osc}_{B(sR)} w_{\nu} = \sum_{\nu=1}^m M_{s\nu}-m_{s\nu}$.
But Theorem \ref{harnack} also applies directly to  $M_{2\nu}-w_{\nu}$ to show:
\[  \left(\frac{1}{V(R)} \int_{B(R)}(M_{2\nu} - w_{\nu}(y))^ pdy \right)^ {1/p} \le C \left(M_{2\nu}-M_{1\nu}+R^2 \sup_{B(2R)} ||\db \dbb h||\right) \]
Adding the two previous inequalities and summing them for $\nu=1 \ldots m$ yields:
\begin{equation}
\label{ine:fin}
\om(R) \le \delta \om(2R)+  R^ {\alpha} ||h||_{\mathscr{C}^ {\alpha, \tilde \beta}(B(2R))}+R^2 \sup_{B(2R)} ||\db \dbb h||
\end{equation}
for $\delta =1-1/C$. Using now a standard lemma (cf \cite[Lemma (8.23)]{Gilb}), we infer that $\om(R) \le C R^{\alpha'}$ for some constant $C>0$ and exponent $\alpha'>0$. \\

$\bullet$
We proved that for all $\gamma$ the oscillation of $(\pi^* \vp)_{\gamma \bar \gamma}$ on any ball $B(R)$ (for $\omb$) as before is dominated by $R^{\alpha}$ for some $\alpha>0$.
Thanks to Lemma \ref{lem:func}, this is equivalent to saying that $||\vp||_{\cdab(B(R))} <+\infty$ for all balls $B(R)$ for the cone metric. Now, take two points $p,q$ on $\DD$, and set $R=d_{\beta}(p,q)$. There are two possibilities: either $R< \frac{1}{2}\min( d_{\beta}(p, \Delta), d_{\beta}(q, \Delta))$ in which case $p,q$ belong to $B(p,2R)$ a geodesic ball and we have the desired estimate for $p,q$. Or we are in the second case, and then we pick a point in $\Delta$ (call it $0$) such that for instance $
d_{\beta}(p,0)\le 2R$. Then $d_{\beta}(q,0) < 3R$; so that $p,q$ belong to $B(0, 3R)$ and we can also apply the previous result. So in the end, we showed that $\vp $ is in the class $\cdab$, so that Theorem B is established whenever the coefficients $\beta_k$ are rational numbers. 

\begin{rema}
In reality, to prove that $\vp \in \cdab$ we also need to show that $\d \vp \in \cab$. Let us briefly mention how this can be done: if $u=\pi^*\vp$ and $\gamma$ is a constant twisted vector field on $ \CC^n$, then we know that $(\Delta_{\om_u}-\mu) u_{\gamma} = F_{\gamma}$, so that Harnack inequality (which clearly also holds for the operator $\Delta_{\om_u}-\mu$) yields as in \cite[Theorem 8.22]{Gilb} the Hölder continuity of $u_{\gamma}$ \--- here we do not need Evans-Krylov's argument because $u_{\gamma}$ satisfies an elliptic \textit{equation}.
\end{rema}

\subsection{The general case of real coefficients: end of the proof of Theorem B}

In this section, we use a density argument to obtain Theorem B without the rationality assumptions on the angles. 

\noindent
So we start from a Monge-Ampère equation:
\[(\om+\ddc \vp) ^n = \frac{e^{\mu \vp+f}dV}{\prod_{j=1}^d |s_j|^{2(1-\beta_j)}}\] 
where the $\beta_j$ are now real numbers in $(0,1)$. We approximate the angles $\beta_j$ by rational numbers $r_{j,k}:=\frac{p_{j,k}}{q_{j,k}}$, and we look at the (renormalized if $\mu=0$) equation
\[(\om+\ddc \vp_k) ^n = \frac{e^{\mu \vp+f}dV}{\prod_{j=1}^d |s_j|^{2(1-r_{j,k})}}\] 
By \cite{DDGHKZ}, the $\mathscr C^ {\alpha}$ norm of $\vp_k$ is uniformly bounded (hence so is its $\cab$ norm), and we know from \cite{Kolo2} that $(\vp_k)_k$ converges toward $\vp$. By Remark \ref{rem:unif}, we have a uniform laplacian estimate; namely $\ddc \vp_k$ is uniformly quasi-isometric to the model cone metric with angles $2\pi r_{j,k}$ along $(z_j=0)$, and in particular $\vp_k$ converges to $\vp$ in $\mathscr C^{\infty}_{\rm loc}(X\setminus \Delta)$. If we show that $\vp_k$ satisfies a uniform $\mathscr{C}^ {2, \alpha, r_k}$ estimate (say near each point lying on $\Delta$), then we will be done.

But we observed (cf Remark \ref{rem:har}) that Harnack inequality was valid uniformly in $k$. Therefore, if we put together the uniform estimates mentionned above with Lemma \ref{lem:func} and inequality \eqref{ine:fin}, we get:
\[||\vp_k||_{\mathscr C^{2,\alpha, r_k}} \le C\]
As $\vp_k$ converges smoothly to $\vp$ on the compact sets of $\DD$, this shows that $\vp \in \cdab$ by the very definition of these functional spaces.

%\backmatter

\bibliographystyle{smfalpha}
\bibliography{biblio}

\end{document}